\documentclass[11pt]{amsart}
\usepackage{amssymb,amsfonts,amsthm,amscd,amsmath,amsgen,upref}

\usepackage[margin=1in]{geometry}

\theoremstyle{plain}

\newtheorem{cor}{Corollary}

\newtheorem{prop}[cor]{Proposition}
\theoremstyle{definition}

\numberwithin{cor}{section}
\numberwithin{equation}{section}

\DeclareMathOperator{\tr}{tr}

\DeclareMathOperator{\C}{C}
\DeclareMathOperator{\USC}{USC}
\DeclareMathOperator{\LSC}{LSC}
\DeclareMathOperator{\BUC}{BUC}
\DeclareMathOperator{\Lip}{Lip}

\newcommand{\ue}{u^\epsilon_k}
\newcommand{\vd}{v^\delta_k}
\newcommand{\wk}{w^\delta_k}

\newcommand{\abs}[1]{\lvert#1\rvert}
\newcommand{\norm}[1]{\lVert#1\rVert}

\def\Xint#1{\mathchoice
{\XXint\displaystyle\textstyle{#1}}%
{\XXint\textstyle\scriptstyle{#1}}%
{\XXint\scriptstyle\scriptscriptstyle{#1}}%
{\XXint\scriptscriptstyle\scriptscriptstyle{#1}}%
\!\int}
\def\XXint#1#2#3{{\setbox0=\hbox{$#1{#2#3}{\int}$ }
\vcenter{\hbox{$#2#3$ }}\kern-.6\wd0}}

\def\dashint{\Xint-}

\title{Stochastic Homogenization of Monotone Systems of Viscous Hamilton-Jacobi Equations with Convex Nonlinearities}

\author{Benjamin J. Fehrman}

\date{May 7, 2012}



\address{Department of Mathematics, The University of Chicago, 5734 S. University Avenue, Chicago IL, 60637.}

\email{bfehrman@math.uchicago.edu}

\begin{document}

\begin{abstract}
We consider the homogenization of monotone systems of viscous Hamilton-Jacobi equations with convex nonlinearities set in the stationary, ergodic setting.  The primary focus of this paper is on collapsing systems which, as the microscopic scale tends to zero, average to a deterministic scalar Hamilton-Jacobi equation.  However, our methods also apply to systems which do not collapse and, as the microscopic scale tends to zero, average to a deterministic system of Hamilton-Jacobi equations.
\end{abstract}

\maketitle

\section{Introduction}

In this paper we study the limiting behavior, as $\epsilon\rightarrow 0$, of the solutions \begin{equation*}\label{solution} u^\epsilon=\left(u^\epsilon_1,\ldots,u^\epsilon_m\right):\mathbb{R}^n\times[0,\infty)\times\Omega\rightarrow\mathbb{R}^m\end{equation*} of degenerate elliptic, monotone systems of the form \begin{equation}\label{col}\left\{\begin{array}{ll} u^\epsilon_{k,t}-\epsilon\tr(A_k(\frac{x}{\epsilon},\omega)D^2u^\epsilon_k)+H_k(Du^\epsilon_k,u^\epsilon,\frac{u^\epsilon_k-u^\epsilon_j}{\epsilon},\frac{x}{\epsilon},\omega)=0 & \textrm{on}\;\;\mathbb{R}^n\times(0,\infty),\\ u^\epsilon=u_0 & \textrm{on}\;\;\mathbb{R}^n\times\left\{0\right\},\end{array}\right.\end{equation} where to simplify notation we write $\frac{u^\epsilon_k-u^\epsilon_j}{\epsilon}$ for the vector $\left(\frac{u^\epsilon_k-u^\epsilon_1}{\epsilon},\ldots,\frac{u^\epsilon_k-u^\epsilon_m}{\epsilon}\right)\in\mathbb{R}^{m-1}$.

The Hamiltonians $H_k=H_k(p,r,s,y,\omega)$ and the diffusion matrices $A_k=A_k(y,\omega)$ are random processes depending on an underlying probability space $(\Omega,\mathcal{F},\mathbb{P})$.  The intuition is that $\omega\in\Omega$ indexes the collection of all systems like (\ref{col}).  We postpone the precise assumptions until Section 2, but remark here that the $A_k$'s and $H_k$'s are stationary ergodic in $(y,\omega)$ and the $H_k$'s are convex and coercive in $p$ and $s$.  Due to the stationarity and ergodicity, the $u^\epsilon$'s ``see'' the entirety of the systems indexed by $\Omega$ and, as $\epsilon\rightarrow 0$, average out to a deterministic limit which, in view of the coercivity, must be scalar.  There is also an initial boundary layer which forces the limiting solution to satisfy a new initial data.

The result is that, as $\epsilon\rightarrow 0$, the $u^\epsilon$'s converge almost surely to a deterministic scalar function which solves a deterministic Hamilton-Jacobi equation for an appropriate initial condition.  More precisely, we identify a deterministic Hamiltonian $\overline{H}:\mathbb{R}^n\times\mathbb{R}\rightarrow\mathbb{R}$ such that, as $\epsilon\rightarrow 0$, for each $k\in\left\{1,\ldots,m\right\}$ and almost surely in $\Omega$, \begin{equation}\label{introcon}\begin{array}{ll} u^\epsilon_k\rightarrow\overline{u} & \textrm{locally uniformly on}\;\;\;\mathbb{R}^n\times(0,\infty),\end{array}\end{equation} where $\overline{u}$ is the solution of the scalar initial value problem \begin{equation}\label{introeq} \left\{\begin{array}{ll}\overline{u}_t+\overline{H}(D\overline{u},\overline{u})=0 & \textrm{on}\;\;\;\mathbb{R}^n\times(0,\infty), \\ \overline{u}=\underline{u}_0 & \textrm{on}\;\;\;\mathbb{R}^n\times\left\{0\right\},\end{array}\right.\end{equation} with $\underline{u}_0:\mathbb{R}^n\rightarrow\mathbb{R}$ the point-wise minimum, for $u_0=(u_{1,0},\ldots,u_{m,0})$, \begin{equation*}\label{introinit}\underline{u}_0=\min_{1\leq k\leq m}u_{k,0}.\end{equation*}

Concrete examples of systems like (\ref{col}) and (\ref{nocol}) below are \begin{equation*}\label{neutrons} u^\epsilon_{k,t}-\epsilon\Delta u^\epsilon_k+\abs{Du^\epsilon_k}^2+\sum_{i\neq k}c_{ki}(\frac{x}{\epsilon},\omega)e^{(u^\epsilon_k-u^\epsilon_i)/\epsilon}=V_k(\frac{x}{\epsilon},\omega),\end{equation*} which arises in the study of neutron transport, e.g., Armstrong and Souganidis \cite{AS1}, as well as coupled systems like \begin{equation*}\label{couple} u^\epsilon_{k,t}-\epsilon d_k(\frac{x}{\epsilon},\omega)\Delta \ue+H_k(Du^\epsilon_k, u^\epsilon,\frac{x}{\epsilon},\omega)+\sum_{i\neq k}c_{k,i}(\frac{x}{\epsilon},\omega)(\frac{u^\epsilon_k-u^\epsilon_i}{\epsilon})_+^2=0,\end{equation*}  and systems of switching games like \begin{equation*}\label{switching} u^\epsilon_{k,t}+\min\left\{\max\left\{-\epsilon\Delta\ue+G_k(D\ue, u^\epsilon_k, \frac{x}{\epsilon},\omega), u^\epsilon_k-M_k(u^\epsilon,\frac{x}{\epsilon},\omega)\right\},u^\epsilon_k-N_k(u^\epsilon, \frac{x}{\epsilon},\omega)\right\}=0, \end{equation*} with $M_k(u^\epsilon,\frac{x}{\epsilon},\omega)=\min\left\{\;u^\epsilon_i-g_{k,i}(\frac{x}{\epsilon},\omega)\;|\;j\neq k\;\right\}$ and $N_k(u^\epsilon, \frac{x}{\epsilon},\omega)=\max\left\{\;u^\epsilon_i-h_{k,i}(\frac{x}{\epsilon},\omega)\;|\;j\neq k\;\right\}$,  which were studied, for $\epsilon=1$, by Engler and Lenhart \cite{EL} and Lenhart \cite{L} and Capuzzo-Dolcetta and Evans \cite{CDE}, Yamada \cite{Y}, Ishii and Koike \cite{IK2}, Lenhart and Yamada \cite{LY} and Lenhart and Belbas \cite{LB} respectively.

The identification of $\overline{H}$ and proof of homogenization follow the methods of Armstrong and Souganidis \cite{AS1,AS}, with more references given later in the introduction.  We use, for each $(p,r)\in\mathbb{R}^n\times\mathbb{R}$, $\omega\in\Omega$  and $\delta>0$, the approximate macroscopic system \begin{equation*}\label{exacell}\begin{array}{ll} \delta\vd-\tr(A_k(y,\omega) D^2\vd)+H_k(p+D\vd,\hat{r},\vd-v^\delta_j,y,\omega)=0 & \textrm{on}\;\;\mathbb{R}^n \end{array}\end{equation*}  with $\hat{r}=(r,\ldots,r)\in\mathbb{R}^m$, to characterize $\overline{H}(p,r)$, almost surely, by \begin{equation*}\label{exaconv} \overline{H}(p,r)=\limsup_{\delta\rightarrow 0}-\delta v^\delta_1(0,\omega).\end{equation*}

We then use the asymptotic properties of the solution $m_\mu=(m_{1,\mu},\ldots,m_{m,\mu})$ of the so called ``metric system,'' \begin{equation*}\label{exametric} \left\{\begin{array}{ll} -\tr(A_k(y,\omega)D^2m_{k,\mu})+H_k(p+Dm_{k,\mu},\hat{r}, m_{k,\mu}-m_{j,\mu},y,\omega)=\mu & \textrm{on}\;\;\;\mathbb{R}^n\setminus D, \\ m_\mu=0 & \textrm{on}\;\;\;\partial D, \end{array}\right. \end{equation*}  for $D$ either a ball or a point (See Section 6 for details), to prove that, almost surely in $\Omega$ and for each $R>0$, \begin{equation}\label{exaupgrade} \lim_{\delta\rightarrow 0}\sup_{y\in B_{R/\delta}}\abs{\overline{H}(p,r)+\delta v^\delta_1(y,\omega)}=0. \end{equation}

If there exist estimates, uniform in $\epsilon$, for the $\abs{u^\epsilon_{k,t}}$ and $\abs{Du^\epsilon_k}$, then (\ref{exaupgrade}) is sufficient to apply the perturbed test function method and conclude the proof.  However, such estimates are not available in general, unless the initial condition $u_0$ satisfies \begin{equation*}\label{introscalar}\begin{array}{ll} u_{i,0}=u_{j,0} & \textrm{for each}\;\;\; i,j\in\left\{1,\ldots,m\right\}.\end{array}\end{equation*}  We therefore must use additional tools from the theory of viscosity solutions to identify the correct initial condition $\underline{u}_0$ and prove the convergence, as $\epsilon\rightarrow 0$, of the $u^\epsilon_k$'s.

A simple example of the above difficulty is illustrated by the system, for $m=2$, \begin{equation}\label{introex}\left\{\begin{array}{ll} u^\epsilon_{k,t}-\epsilon\Delta\ue+\abs{D\ue}^2+\left(\frac{u^\epsilon_k-u^\epsilon_j}{\epsilon}\right)_+^2=0 & \textrm{on}\;\;\;\mathbb{R}^n\times(0,\infty), \\ u^\epsilon_1=1, u^\epsilon_2=0 & \textrm{on}\;\;\;\mathbb{R}^n\times\left\{0\right\},\end{array}\right. \end{equation} with solution \begin{equation*}\label{introex1}\begin{array}{lll}  u^\epsilon_1(x,t)= (1+t/\epsilon^2)^{-1} & \textrm{and} & u^\epsilon_2(x,t)=0. \end{array}\end{equation*}  It is immediate that, as $\epsilon\rightarrow 0$, $u^\epsilon_{1,t}$ is unbounded from below and, for $k=1,2$, \begin{equation*}\label{introex2}\begin{array}{ll} u^\epsilon_k\rightarrow 0 & \textrm{locally uniformly on}\;\;\;\mathbb{R}^n\times (0,\infty).\end{array}\end{equation*}  Notice that $\overline{u}=0$ is the the unique solution of the limiting problem \begin{equation*}\label{introex3}\left\{\begin{array}{ll} \overline{u}_t+\abs{D\overline{u}}^2=0 & \textrm{on}\;\;\;\mathbb{R}^n\times (0,\infty),\\ \overline{u}=\underline{u}_0=0 & \textrm{on}\;\;\;\mathbb{R}^n\times\left\{0\right\}.\end{array}\right.\end{equation*}  We prove that the analogous behavior is replicated for general such systems in the random setting.

Finally, we remark that our methods are easily adapted to systems of the form \begin{equation}\label{nocol}\left\{\begin{array}{ll} u^\epsilon_{k,t}-\epsilon\tr(A_k(\frac{x}{\epsilon},\omega)D^2u^\epsilon_k)+H_k(Du^\epsilon_k,u^\epsilon,\frac{x}{\epsilon},\omega)=0 & \textrm{on}\;\;\mathbb{R}^n\times (0,\infty), \\ u^\epsilon=u_0 & \textrm{on}\;\;\mathbb{R}^n\times\left\{0\right\},\end{array}\right.\end{equation} which, as $\epsilon\rightarrow 0$, do not collapse and homogenize to a deterministic system, as well as the corresponding time-independent analogues of (\ref{col}) and (\ref{nocol}).

The homogenization of scalar equations in stationary ergodic random environments has been studied extensively.  The linear case was first analyzed by Papanicolaou and Varadhan \cite{PV,PV1} and Kozlov \cite{K}, and general variational problems were considered by Del Maso and Modica \cite{DM,DM1}.

More recently, results for Hamilton-Jacobi equations were first obtained by Souganidis \cite{S} and Rezakhanlou and Tarver \cite{RT}, for viscous Hamilton-Jacobi equations by Lions and Souganidis \cite{LS1,LS2} and Kosygina, Rezakhanlou and Varadhan \cite{KRV}, and for viscous Hamilton-Jacobi equations in unbounded environments by Armstrong and Souganidis \cite{AS}.

The homogenization of systems has been most extensively studied in the periodic setting.  The homogenization of linear elliptic systems has been considered, for example, by Avellaneda and Lin \cite{AL}.  Camilli, Ley and Loreti \cite{CLL} considered the homogenization of a first-order monotone system of Hamilton-Jacobi equations.  And, more recently, Mitake and Tran \cite{MT} considered a weakly-coupled, collapsing system of first-order Hamilton-Jacobi equations.

In the stationary, ergodic setting, Armstrong and Souganidis \cite{AS1} considered a time-independent, uniformly-elliptic version of (\ref{col}).  In addition to characterizing the limiting behavior, as $\epsilon\rightarrow 0$, of the system's principle eigenvalue, they prove that if (\ref{introcon}) is known, a priori, then the limit satisfies the time-independent version of (\ref{introeq}).  We prove (\ref{introcon}) in general and obtain a similar result in Proposition \ref{colnotresult}.  However, beginning with the study of the approximate macroscopic system, our analysis differs from \cite{AS1} due to the absence of uniform ellipticity and the existence of an initial boundary layer for general parabolic problems.

The paper is organized as follows.  In Section 2, we introduce the notation, state the precise hypotheses for the coefficients $A_k$ and $H_k$, and recall two ergodic theorems used throughout the paper.  In Section 3, we study the approximate macroscopic problem.   The effective Hamiltonian $\overline{H}$ and its properties are the subjects of Sections 4 and 5.  We study the metric system in Sections 6 and 7.  The effective Hamiltonian is further investigated in Section 8.  We identify the initial data $\underline{u}_0$ and conclude the proof of the main result in Section 9.  We present precise results for (\ref{nocol}) and the time-independent analogues in Section 10.

\subsection*{Acknowledgments}

I would like to thank Scott Armstrong and Panagiotis Souganidis for suggesting this problem and for many useful conversations.  Furthermore, I would like to thank Panagiotis Souganidis for his numerous suggestions and advice throughout the process of writing and editing this paper.

\section{Preliminaries}

\subsection{Notation}

Elements of $\mathbb{R}^n$ and $[0,\infty)$ are denoted by $x$ and $y$ and $t$ respectively and, for $r\in\mathbb{R}$, $\hat{r}=(r,\ldots,r)\in\mathbb{R}^m$.  For $r,s\in\mathbb{R}^l$, we write $r\leq s$ if $r_i\leq s_i$ for each $i\in\left\{1,\ldots,l\right\}$.  We write $Dv$ and $v_t$ for the derivative of the scalar function $v$ with respect to $x\in\mathbb{R}^n$ and $t\in[0,\infty)$, while $D^2v$ stands for the Hessian of $v$.  Regarding the Hamiltonians $H_k$ we write $p$ for the dependence on $D\ue$ and $r$ for the dependence on $u^\epsilon$.  The variable $s\in\mathbb{R}^{m-1}$ is used for the differences $\epsilon^{-1}(u^\epsilon_k-u^\epsilon_j).$  We use the notation $D_pH_k$ for the derivative of $H_k$ with respect to the gradient variable, while other derivatives are expressed analogously.  The spaces of $k\times l$ and $k\times k$ symmetric matrices with real entries are respectively written $\mathcal{M}^{k\times l}$ and $\mathcal{S}(k)$.  If $M\in\mathcal{M}^{k\times l}$, then $M^t$ is its transpose and $\abs{M}$ is its norm $\abs{M}=\tr(MM^t)^{1/2}.$  If $M$ is a square matrix, we write $\tr(M)$ for the trace of $M$.  For $U\subset\mathbb{R}^n$, $\USC(U;\mathbb{R}^d)$, $\LSC(U;\mathbb{R}^d)$, $\BUC(U;\mathbb{R}^d)$, $\Lip(U;\mathbb{R}^d)$ and $\C^k(U;\mathbb{R}^d)$ are the spaces of upper-semicontinuous, lower-semicontinuous, bounded continuous, Lipschitz continuous and $k$-continuously differentiable functions on $U$ with values in $\mathbb{R}^d$.  Moreover, $B_R$ and $B_R(x)$ are respectively the open balls of radius $R$ centered at zero and $x\in\mathbb{R}^n$.  We denote by $\Omega\supset\Omega_1\supset\Omega_2\supset\Omega_3$ nested subsets of full probability.  Finally, throughout the paper we write $C$ for constants that may change from line to line but are independent of $\omega\in\Omega$ unless otherwise indicated.

\subsection{The random medium}

The random medium is described by the probability space $(\Omega,\mathcal{F},\mathbb{P})$.  An element $\omega\in\Omega$ then corresponds to a particular realization of the environment.

It is assumed that $\Omega$ is equipped with a group $\left(\tau_y\right)_{y\in\mathbb{R}^n}$ of transformations $\tau_y:\Omega\rightarrow\Omega$ which are \begin{equation}\label{transgroup} \textrm{measure-preserving and ergodic,}\end{equation} where the latter means that, if $E\subset\Omega$ satisfies $\tau_y(E)=E$ for each $y\in\mathbb{R}^n$ then $\mathbb{P}(E)=0$ or $\mathbb{P}(E)=1$.

A process $f:\mathbb{R}^n\times\Omega\rightarrow\mathbb{R}$ is said to be stationary if the law of $f(y,\cdot)$ is independent of $y\in\mathbb{R}^n$, a property which can be reformulated using  $\left(\tau_y\right)_{y\in\mathbb{R}^n}$ as \begin{equation}\label{statergodic} \begin{array}{ll} f(y+z,\cdot)=f(y,\tau_z\cdot) & \textrm{for all}\;\;\;y,z\in\mathbb{R}^n. \end{array}\end{equation}  To simplify statements we say that a process is stationary ergodic if it satisfies (\ref{statergodic}) and $\left(\tau_y\right)_{y\in\mathbb{R}^n}$ is ergodic.

The following ergodic theorem will be used frequently in this paper.  A proof may be found in Becker \cite{B}.  Here $\mathbb{E}f$ denotes the expectation of a random variable $f$.

\begin{prop}\label{ergodic1}  Assume (\ref{transgroup}) and suppose that $f:\mathbb{R}^n\times\Omega\rightarrow\mathbb{R}$ is stationary and $\mathbb{E}\abs{f(0,\cdot)}<\infty$.  There exists a subset $\tilde{\Omega}\subset\Omega$ such that $\mathbb{P}(\tilde{\Omega})=1$ and, for every bounded domain $V\subset\mathbb{R}^n$ and $\omega\in\tilde{\Omega}$, $$\lim_{t\rightarrow\infty}\dashint_{tV}f(y,\omega)dy=\mathbb{E}f.$$ \end{prop}

The subadditive ergodic theorem is also used in this paper.  A proof may be found in Akcoglu and Krengel \cite{AK}.  Its statement requires more terminology.  Let $\mathcal{I}$ denote the class of subsets of $[0,\infty)$ consisting of finite unions of intervals of the form $[a,b)$ and let $\left(\sigma_t\right)_{t\geq 0}$ be a semigroup of measure-preserving transformations $\sigma_t:\Omega\rightarrow\Omega$.  A map $Q:\mathcal{I}\rightarrow L^1(\Omega,\mathbb{P})$ such that:  \begin{center}\begin{enumerate} \item $Q(I)(\sigma_t\omega)=Q(I+t)\omega\;\;\;\textrm{almost surely in}\;\Omega,$

\item $\mathbb{E}\abs{Q(I)}\leq C\abs{I},$ for some $C>0$ and every $I\in\mathcal{I}$,

\item If $I_1,\ldots,I_k\in\mathcal{I}$ are disjoint then, $Q(\cup_{j=1}^k)\leq \sum_{j=1}^k Q(I_j),$ \end{enumerate} \end{center}  is called a \emph{continuous subadditive process} with respect to the semigroup $\left(\sigma_t\right)_{t\geq 0}$.

\begin{prop}\label{subadditiveergodic}  If $Q$ is a continuous subadditive process with respect to the semigroup $\left(\sigma_t\right)_{t\geq 0}$, there exists a random variable $a$ which is invariant under $\left(\sigma_t\right)_{t\geq 0}$ such that, almost surely, $$\lim_{t\rightarrow\infty}\frac{1}{t}Q([0,t))(\omega)=a(\omega).$$  If $\left(\sigma_t\right)_{t\geq 0}$ is ergodic, then $a$ is constant.\end{prop}

\subsection{The assumptions}

We state below a number of assumptions for the $A_k$'s and $H_k$'s.  Some are necessary to insure the well-posedness of (\ref{col}), while others are crucial for the homogenization, collapse and identification of the appropriate initial condition.

Because every assumption must hold for all $k\in\left\{1,\ldots,m\right\}$, to avoid cumbersome statements, we do not repeat this quantifier for each of the assertions below.  Furthermore, we make the convention that, unless otherwise indicated, each statement holds globally for $p\in\mathbb{R}^n$, $r\in\mathbb{R}^m$, $s\in\mathbb{R}^{m-1}$, $y\in\mathbb{R}^n$ and $\omega\in\Omega$.

For each fixed $(p,r,s)$, \begin{equation}\label{stationary1} \begin{array}{llll} (y,\omega)\rightarrow A_k(y,\omega) & \textrm{and} & (y,\omega)\rightarrow H_k(p,r,s,y,\omega) & \textrm{are stationary.} \end{array}\end{equation}

For each fixed $(r,s,y,\omega)$, \begin{equation}\label{gradconvex}\begin{array}{ll} p\rightarrow H_k(p,r,s,y,\omega) & \textrm{is convex,} \end{array}\end{equation} and for each fixed $(p,r,y,\omega)$,  \begin{equation}\label{difconvex}\begin{array}{ll} s\rightarrow H_k(p,r,s,y,\omega) & \textrm{is convex.}\end{array}\end{equation}

The Hamiltonians $H_k$ are coercive in $p$ and $s$, i.e., for each $R>0$, there exist constants $C_1,C_2,C_3>0$ and $\gamma_k>1$ satisfying, for all $r\in B_R$, \begin{equation}\label{coercive} C_1\abs{p}^{\gamma_k}+C_2\max_{i\neq k}(s_i)_+-C_3\leq H_k(p,r,s,y,\omega).\end{equation}

We also assume that the matrix $A_k$ has a Lipschitz continuous square root in the sense that \begin{equation}\label{matsquare}  A_k(y,\omega)=\Sigma_k(y,\omega)\Sigma_k^t(y,\omega) \end{equation} with, for $C>0$ and every $\omega\in\Omega$, \begin{equation}\label{lipsigma} \norm{\Sigma_k(\cdot,\omega)}_{C^{0,1}(\mathbb{R}^n;M^{n_k\times s_k})}\leq C.\end{equation}

Moreover, for each fixed $(p,y,\omega)$, \begin{equation}\label{hamincrease} H_k(p,r,s,y,\omega)\;\; \textrm{is nondecreasing in}\;\; r_k\;\;\textrm{and}\;\; s_i\end{equation} and \begin{equation}\label{hamdecrease} H_k(p,r,s,y,\omega)\;\;\textrm{is nonincreasing in}\;\; r_i\;\;\textrm{for}\;\; i\neq k.\end{equation}  In addition, the Hamiltonians are bounded for bounded $(p,r,s)$, i.e., for each $R>0$ there exists $C_4=C_4(R)>0$ such that, for all $(p,r,s)\in B_R\times B_R\times B_R$, \begin{equation}\label{bounded} \abs{H_k(p,r,s,y,\omega)}\leq C_4, \end{equation} and, locally in $p$, Lipschitz continuous, i.e., for each $R>0$ there exists $C_5=C_5(R)>0$ such that, for all $(r_i,s_i)\in B_R\times B_R$, \begin{equation}\label{hamcon}
|H_k(p_1,r_1,s_1,y_1,\omega)-H_k(p_2,r_2,s_2,y_2,\omega)|<C_5[(1+|p_1|+|p_2|)^{\gamma_k-1}|p_1-p_2|+\abs{r_1-r_2}+\abs{s_1-s_2}+\abs{y_1-y_2}]. \end{equation}

We remark that (\ref{hamincrease}) and (\ref{hamdecrease}) imply that the Hamiltonian $H_k$ is monotone in the sense that, for $r,q\in\mathbb{R}^m$, if $r\leq q$ and $r_k=q_k$ then, \begin{equation}\label{monotone} H_k(p,r,s,y,\omega)\geq H_k(p,q,s,y,\omega). \end{equation}  The last assumption we need for the $H_k$'s is that, for $r,q\in\mathbb{R}^m$, if $r_k-q_k=\max_i|r_i-q_i|$ then, \begin{equation}\label{colmonstrict} H_k(p,r,s,y,\omega)-H_k(p,q,s,y,\omega)\geq 0. \end{equation}

Finally, we assume that \begin{equation}\label{initialbuc} u_0\in\BUC(\mathbb{R}^n;\mathbb{R}^m).\end{equation}

Among the above, (\ref{coercive}), (\ref{matsquare}), (\ref{lipsigma}), (\ref{hamincrease}), (\ref{hamdecrease}), (\ref{bounded}), (\ref{hamcon}), (\ref{monotone}), (\ref{colmonstrict})  and (\ref{initialbuc}) are necessary for the well-posedness of (\ref{col}) (See Ishii and Koike \cite{IK}). The rest, i.e., (\ref{stationary1}), (\ref{gradconvex}) and (\ref{difconvex}), are necessary for the homogenization and (\ref{coercive}) for the collapse of the system.

Throughout the paper we will assume each of the statements (\ref{transgroup})-(\ref{initialbuc}).  To avoid repeating all of these, we introduce a steady assumption. \begin{equation}\label{steady} \textrm{Assume}\:(\ref{transgroup}),\: (\ref{statergodic}),\: (\ref{stationary1}),\: \ldots,\: (\ref{colmonstrict}),\: (\ref{initialbuc}). \end{equation}

\section{The Macroscopic System}

We study here, for each $(p,r)\in\mathbb{R}^n\times\mathbb{R}$, $\omega\in\Omega$ and $\delta>0$, the solutions $v^\delta=(v^\delta_1,\ldots,v^\delta_m)$ of the approximate macroscopic system \begin{equation}\label{cell}\begin{array}{ll} \delta\vd-\tr(A_k(y,\omega)D^2\vd)+H_k(p+D\vd,\hat{r},\vd-v^\delta_j,y,\omega)=0 & \textrm{on}\;\;\;\mathbb{R}^n, \end{array}\end{equation} where $\hat{r}=(r,\ldots,r)\in\mathbb{R}^m$.

\begin{prop}\label{cellsol}  Assume (\ref{steady}).  For each $(p,r)\in\mathbb{R}^n\times\mathbb{R}$, $\delta>0$ and $\omega\in\Omega$, there exists a unique solution $v^\delta$ of (\ref{cell}) such that, for $C=C(p,r,n)>0$, \begin{equation}\label{gradbound} \begin{array}{lll} \max_{1\leq k\leq m}\norm{\delta\vd}_{L^\infty(\mathbb{R}^n)}\leq C & \textrm{and} &  \max_{1\leq k\leq m} \norm{D\vd}_{L^\infty(\mathbb{R}^n)}\leq C. \end{array}\end{equation}  Furthermore, for each $R>0$, there exists $C=C(p,r,n,R)>0$ such that $$\max_{1\leq j,k\leq m}\norm{\vd-v^\delta_j}_{L^\infty(B_R)}\leq C.$$  Finally, the process $v^\delta:\mathbb{R}^n\times\Omega\rightarrow\mathbb{R}^m$ is stationary in the sense of (\ref{stationary1}).\end{prop}

\begin{proof}  The existence and uniqueness of a solution for (\ref{cell}), for each $(p,r)\in\mathbb{R}^n\times\mathbb{R}$, $\delta>0$ and $\omega\in\Omega$, follow from the standard comparison result and Perron's method (See \cite{IK}).  The stationarity of $v^\delta$ is an immediate consequence of (\ref{stationary1}) and the uniqueness.  Finally, the estimates for the $\delta v^\delta_k$'s and $Dv^\delta_k$'s are obtained almost exactly as in \cite{AS1,AS,LS1}.

The new part of the argument is the inequality on the differences.  Choose a function $\psi\in C^\infty(\mathbb{R}^n)$ satisfying $0\leq \psi \leq 1$ with $\psi=1$ on $\overline{B}_1$ and $\psi=0$ on $\mathbb{R}^n\setminus B_{2}$.  Define $\phi=\psi^4$ and observe, for $C>0$,\begin{equation}\label{auxcut1}  \begin{array}{lll} \abs{D\phi}^2\leq C\phi^{\frac{3}{2}} & \textrm{and} & \abs{D^2\phi}\leq C\phi^\frac{1}{2}. \end{array}\end{equation}

Fix $R\geq 1$, $\alpha\geq 1$ and $k\in\left\{1,\ldots,m\right\}$.  Let $\phi_R(\cdot)=\phi(\frac{\cdot}{R})$ and observe that $\phi_R$ satisfies (\ref{auxcut1}) with the same constant.  We multiply the $k$-th component of (\ref{cell}) by the test function $$w^\delta_k=\phi_R(\vd-\min_{1\leq i\leq m} v^\delta_i)^\alpha$$ and integrate over $\mathbb{R}^n$.

Using (\ref{coercive}), we have $$\begin{array}{ll} C_2(v^\delta_k-\min_{1\leq i\leq m}v^\delta_i)-C_3\leq H_k(p+Dv^\delta_k,\hat{r},v^\delta_k-v^\delta_j,y,\omega) & \textrm{on}\;\;\;\mathbb{R}^n\end{array}$$ and, hence, using the equivalence of viscosity and distributional solutions for linear inequalities (See Ishii \cite{I}), \begin{equation*}\label{auxcol1} \int_{\mathbb{R}^n}(\delta\vd)w^\delta_k dy+\int_{\mathbb{R}^n}v^\delta_{k,x_i}(a_{ij}w^k_\delta)_{x_j} dy+C_2\int_{\mathbb{R}^n}\phi_R(\vd-\min_{1\leq i\leq m} v^\delta_i)^{\alpha+1} dy \leq C_3\int_{\mathbb{R}^n} w^\delta_k dy. \end{equation*}

In view of (\ref{gradbound}) and (\ref{auxcut1}), there exists $C>0$ satisfying $$\begin{array}{lll} \abs{Dw^\delta_k} & = & \abs{D\phi_R(v^\delta_k-\min_{1\leq i\leq m}v^\delta_i)^\alpha+\phi_R\alpha(v^\delta_k-\min_{1\leq i\leq m}v^\delta_i)^{\alpha-1}D(v^\delta_k-\min_{1\leq k\leq m}v^\delta_i)} \\ & \leq & C(\alpha\phi_R(v^\delta_k-\min_{1\leq i\leq m}v^\delta_i)^{\alpha-1}+\phi_R^{\frac{3}{4}}(v^\delta_k-\min_{1\leq i\leq m}v^\delta_i)^\alpha).\end{array}$$  Therefore, it follows from (\ref{matsquare}) and (\ref{lipsigma}) that, for $C=C(R,\alpha)>0$, \begin{equation*}\label{auxcol2} \int_{\mathbb{R}^n}\phi_R(\vd-\min_{1\leq i\leq m} v^\delta_i)^{\alpha+1} dy \leq C\int_{\mathbb{R}^n}\phi_R^{\frac{3}{4}}(\vd-\min_{1\leq i\leq m} v^\delta_i)^\alpha dy +C\int_{\mathbb{R}^n}\phi_R(\vd-\min_{1\leq i\leq m} v^\delta_i)^{\alpha-1} dy. \end{equation*}

We use H\"older's inequality and Cauchy's inequality to conclude that, for $C=C(\alpha,R)>0$, \begin{equation*}\label{auxcol3} \int_{B_R} (\vd-\min_{1\leq i\leq m} v^\delta_i)^{\alpha+1} dy \leq \int_{\mathbb{R}^n}\phi_R(\vd-\min_{1\leq i\leq m} v^\delta_i)^{\alpha+1} dy\leq C. \end{equation*}

Fix $\alpha=2n-1$.  In view of (\ref{gradbound}), Morrey's inequality yields, for $C=C(R,n)>0$, \begin{equation*}\label{auxcol4} \norm{\vd-\min_{1\leq i\leq m} v^\delta_i}_{L^\infty(B_R)}\leq \norm{\vd-\min_{1\leq i\leq m} v^\delta_i}_{C^{0,\frac{1}{2}}(B_R)}\leq C.\end{equation*}  Since $k\in\left\{1,\ldots,m\right\}$ was arbitrary, we conclude using the triangle inequality.  \end{proof}

In the following proposition, we show that the solutions $v^\delta$ depend continuously on $(p,r)\in\mathbb{R}^n\times\mathbb{R}$.  Its proof is identical to (\ref{effcon}), and is therefore omitted.  We remark that, because the estimates appearing in Proposition \ref{cellsol} depend on $(p,r)$, we obtain here only local estimates.

\begin{prop}\label{contdep}  Assume (\ref{steady}).  For each $R>0$ and $\omega\in\Omega$, there exists $C=C(R)>0$ such that, if $v^{i,\delta}$ is the solution of (\ref{cell}) corresponding to $(p_i,r_i)\in B_R\times B_R$, then $$\max_{1\leq k\leq m}\norm{\delta v^{1,\delta}_k-\delta v^{2,\delta}_k}_{L^\infty(\mathbb{R}^n)}\leq C\max_{1\leq k\leq m}\left((1+\abs{p_1}+\abs{p_2})^{\gamma_k-1}\abs{p_1-p_2}+\abs{r_1-r_2}\right).$$\end{prop}

\section{The Effective Hamiltonian}

In this section, we construct the deterministic Hamiltonian $\overline{H}(p,r)$ and obtain, on a subset of full probability, a subsolution which grows sublinearly at infinity of the macroscopic system \begin{equation}\label{cellsub}\begin{array}{ll} -\tr(A_k(y,\omega) D^2w_k)+H_k(p+Dw_k,\hat{r},w_k-w_j,y,\omega)=\overline{H}(p,r) & \textrm{on}\;\;\;\mathbb{R}^n. \end{array}\end{equation}

Define $$\overline{H}(p,r,\omega)=\limsup_{\delta\rightarrow 0}-\delta v^\delta_1(0,\omega)$$ for $v^\delta$ the solution of (\ref{cell}) corresponding to $(p,r)$, and observe that Propositions \ref{cellsol} and \ref{contdep} and (\ref{transgroup}) imply $\overline{H}(p,r,\omega)=\overline{H}(p,r)$ is deterministic, satisfying, on a subset full probability, \begin{equation}\label{effectiveham1} \overline{H}(p,r)=\limsup_{\delta\rightarrow 0}-\delta v^\delta_1(0,\omega).\end{equation}  We remark that the choice $-\delta v^\delta_1(0,\omega)$ in (\ref{effectiveham1}) is arbitrary but, in view of Proposition \ref{cellsol}, it does not effect the definition of $\overline{H}$.

Let \begin{equation}\label{auxnorm}\begin{array}{lll} w^\delta=(w^\delta_1,\ldots,w^\delta_m)  & \textrm{with} & \wk(y,\omega)=\vd(y,\omega)-v^\delta_1(0,\omega)\end{array} \end{equation} for $v^\delta$ the solution of (\ref{cell}) corresponding to $(p,r)$, and observe that $w^\delta$ satisfies the system \begin{equation}\label{cellnorm}\begin{array}{ll} \delta\wk-\tr(A_k(y,\omega) D^2\wk)+H_k(p+D\wk,\hat{r},\wk-w^\delta_j,y,\omega)=-\delta v^\delta_1(0,\omega) & \textrm{on}\;\;\;\mathbb{R}^n.\end{array} \end{equation}  The following proposition is an immediate consequence of the uniform estimates obtained in Proposition \ref{cellsol}.

\begin{prop}\label{cellnormsol}  Assume (\ref{steady}).  For each $(p,r)\in\mathbb{R}^n\times\mathbb{R}$ and $\omega\in\Omega$, there exists $C=C(p,r,n)>0$ satisfying, for each $R>0$, $$\begin{array}{lll} \max_{1\leq k\leq m}\norm{D\wk}_{L^\infty(\mathbb{R}^n)}\leq C & \textrm{and} & \max_{1\leq k\leq m}\norm{w^\delta_k}_{L^\infty(B_R)}<C(1+R).\end{array}$$ \end{prop}

We now obtain the desired subsolution of (\ref{cellsub}) by passing to the limit, as $\delta\rightarrow 0$, in (\ref{cellnorm}).  More precisely, we use the convexity of the $H_k$'s and the equivalence of viscosity and distributional solutions for linear inequalities to pass weakly to the limit in (\ref{cellnorm}) and obtain, on a subset of full probability, a subsolution of (\ref{cellsub}) which grows sublinearly at infinity.

The following proposition is used to characterize the behavior at infinity.  Its proof is a consequence of the ergodic theorem and may be found in the appendix of \cite{AS}.

\begin{prop}\label{statsub}  Let $w:\mathbb{R}^n\times\Omega\rightarrow\mathbb{R}$ and $W:\mathbb{R}^n\times\Omega\rightarrow \mathbb{R}^n$ satisfy $W(\cdot,\omega)=Dw(\cdot,\omega)$ almost surely in the sense of distributions.  Assume that $W$ is stationary satisfying $\mathbb{E}W(0,\cdot)=0$ and $W(0,\cdot)\in L^\alpha(\Omega)$ for some $\alpha>n$.  Then, almost surely, \begin{equation}\label{atinfinity} \lim_{\abs{y}\rightarrow\infty}\frac{w(y,\omega)}{\abs{y}}=0.\end{equation} \end{prop}

\begin{prop}\label{subsolution}  Assume (\ref{steady}).  There exists a subset $\Omega_1\subset\Omega$ of full probability such that, for every $(p,r)\in\mathbb{R}^n\times\mathbb{R}$, (\ref{cellsub}) admits a subsolution $w=(w_1,\ldots,w_m)$ satisfying, for each $k\in\left\{1,\ldots,m\right\}$ and $C=C(p,r,n)>0$, \begin{equation}\label{subsolutionest}\begin{array}{llll} \lim_{\abs{y}\rightarrow\infty}\frac{w_k(y,\omega)}{\abs{y}}=0, & \norm{Dw_k}_{L^
\infty(\mathbb{R}^n)}\leq C & \textrm{and} & \norm{w_k}_{L^\infty(B_R)}\leq C(1+R).\end{array}\end{equation}\end{prop}

\begin{proof}  The proof is nearly identical to the analogous arguments in \cite{AS1,AS,LS2}.  We therefore only sketch each step.  To simplify notation, we write $w^\delta=(w^\delta_1,\ldots,w^\delta_m)\in L^\infty_{\textrm{loc}}(\mathbb{R}^n)^m$ and $Dw^\delta=(Dw^\delta_1,\ldots,Dw^\delta_m)\in L^\infty(\mathbb{R}^n;\mathbb{R}^{n})^m$.

For each $\delta>0$, $w^\delta$ is a distributional solution of (\ref{cellnorm}) (See \cite{I}).  We now pass weakly to the limit, as $\delta\rightarrow 0$, to obtain a subsolution of (\ref{cellsub}).

Using Proposition \ref{cellnormsol} and (\ref{transgroup}), there exists a deterministic $\tilde{H}(p,r)\in L^\infty(\Omega)$ and, for each $R>0$,  there exist $w=(w_1,\ldots,w_m)\in L^\infty(B_R\times\Omega)^m$ and $W=(W_1,\ldots,W_m)\in L^\infty(B_R\times\Omega;\mathbb{R}^n)^m$ such that, after passing to a subsequence, \begin{equation}\label{auxconvergence1}\begin{array}{lll} -\delta v^\delta_1(0,\omega)\rightharpoonup\tilde{H}(p,r) & \textrm{in} & L^\infty(\Omega)\;\;\textrm{weak-*}\end{array}\end{equation} and \begin{equation}\label{auxconvergence}\begin{array}{lll} (w^\delta, Dw^\delta) \rightharpoonup (w, W) & \textrm{in} & L^\infty(B_R\times\Omega)^m\times L^\infty(B_R\times\Omega;\mathbb{R}^n)^m\;\;\textrm{weak-*}\end{array}\end{equation}where, on a subset of full probability and for each $k\in\left\{1,\ldots,m\right\}$, $W_k=Dw_k$ in the sense of distributions (See \cite{AS}).  Moreover, (\ref{auxconvergence}) implies that, for each $R>0$ and any $1\leq p<\infty$, there exists a sequence of convex combinations of the $(w^\delta,Dw^\delta)$, depending on $p$, which converges strongly to $(w,Dw)$ in $L^p(B_R\times\Omega)^m\times L^p(B_R\times\Omega;\mathbb{R}^n)^m$.

The convexity (\ref{gradconvex}) and (\ref{difconvex}) and Proposition \ref{cellnormsol}, (\ref{auxconvergence1}) and (\ref{auxconvergence}) imply using the Dominated Convergence Theorem that, on a subset of full probability, $w=(w_1,\ldots,w_m)$ is a distributional solution of the system \begin{equation}\label{subsolution5} -\tr(A_k(y,\omega)D^2w_k)+H_k(p+Dw_k,\hat{r},w_k-w_j,y,\omega)\leq \tilde{H}(p,r)\leq \overline{H}(p,r). \end{equation}  The convexity (\ref{gradconvex}) and (\ref{difconvex}) and the equivalence of distributional and viscosity solutions for linear inequalities imply that $w$ is a viscosity solution of (\ref{cellsub}) (See \cite{I}).

The estimates (\ref{subsolutionest}) are an immediate consequence of (\ref{auxconvergence}) and Proposition \ref{cellsol}.  Furthermore, since $v^\delta$ is stationary we have, for each $\delta>0$ and $k\in\left\{1,\ldots,m\right\}$, $$\begin{array}{lll} \mathbb{E}(D v^\delta_k(0,\cdot))=0 & \textrm{and, therefore,} & \mathbb{E}(Dw_k(0,\cdot))=0.\end{array}$$  We conclude using Proposition \ref{statsub} that there exists a subset of full probability on which, for each $k\in\left\{1,\ldots,m\right\}$, $\lim_{\abs{y}\rightarrow\infty}\abs{y}^{-1}w_k(y,\omega)=0$.

Therefore, there exists a subset $\Omega_1(p,r)\subset\Omega$ of full probability such that $w$ is a subsolution of (\ref{cell}) corresponding to $(p,r)$ satisfying (\ref{subsolutionest}).  To conclude, define $\Omega_1=\bigcap_{(p,r)\in\mathbb{Q}^n\times\mathbb{Q}}\Omega_1(p,r)$ and apply Proposition \ref{contdep}. \end{proof}

We conclude this section by strengthening our characterization of $\overline{H}(p,r)$.  We will use crucially that, on a subset of full probability, the constructed subsolution $w$ of (\ref{cellsub}) grows sub-linearly at infinity.

\begin{prop}\label{auxenhance}  Assume (\ref{steady}).  For each $(p,r)\in\mathbb{R}^n\times\mathbb{R}$ and $R>0$, \begin{equation}\label{charac1} \lim_{\delta \rightarrow 0}\sup_{y\in B_{R/\delta}}\mathbb{E}(\abs{\overline{H}(p,r)+\delta v^{\delta}_1(y,\omega)})=0.\end{equation} \end{prop}

\begin{proof}  We first show, using the notation of Proposition \ref{subsolution}, that $\tilde{H}(p,r)=\overline{H}(p,r)$.  Fix $(p,r)\in\mathbb{R}^n\times\mathbb{R}$ and let $v^\delta$ denote the solution of (\ref{cell}) corresponding to $(p,r)$.  In view of (\ref{auxconvergence1}), \begin{equation}\label{enhance1} \tilde{H}(p,r)\leq \overline{H}(p,r).\end{equation}  The opposite inequality is obtained by a comparison argument.

Define $$\varphi(y)=-(1+\abs{y}^2)^{1/2}+1$$  and introduce, for each $\omega\in\Omega$, $\delta>0$ and $\epsilon>0$, $$\begin{array}{lll} z^\delta=(z^\delta_1,\ldots,z^\delta_m) & \textrm{with} &  z^\delta_k=w_k-\delta^{-1}(\tilde{H}(p,r)+\eta)+\epsilon\varphi. \end{array}$$

In view of (\ref{gradconvex}), (\ref{difconvex}), (\ref{matsquare}), (\ref{lipsigma}), (\ref{hamcon}) and (\ref{subsolution5}), there exists $C>0$ such that, for every $\omega\in\Omega_1$, the function $z^\delta$ is a solution of the inequality \begin{equation}\label{effectcomp} \delta z^
\delta_k-\tr(A_k(y,\omega)D^2 z^\delta_k)+H_k(p+D z^\delta_k,\hat{r},z^\delta_k-z^\delta_j,y,\omega)\leq \delta z^\delta_k+\tilde{H}(p,r)+C\epsilon.\end{equation}

For each $\omega\in\Omega_1$ and $k\in\left\{1,\ldots,m\right\}$, $$\begin{array}{lll} \lim_{|y|\rightarrow\infty}\frac{\vd(y,\omega)}{\abs{y}}=0 & \textrm{and} & \lim_{|y|\rightarrow\infty}\frac{ z^\delta_k(y,\omega)}{\abs{y}}=-1  \end{array}$$ and, hence, for all $\delta>0$ and $\omega\in\Omega_1$, there exists $R(\delta,\omega)>0$ such that $z^\delta\leq v^\delta$ on $\partial B_{R(\delta,\omega)}$.  Furthermore, for each $\omega\in\Omega_1$ and $\epsilon>0$, there exists $C_\epsilon=C_\epsilon(\omega) >0$ such that $$\begin{array}{ll} \max_{1\leq k\leq m} w_k(y,\omega)\leq \epsilon\abs{y}+C_\epsilon & \textrm{on}\;\;\;\mathbb{R}^n.\end{array}$$

The righthand side of (\ref{effectcomp}) is therefore bounded by $$\delta z^\delta_k-\tr(A_k(y,\omega)D^2 z^\delta_k)+H_k(p+D z^\delta_k,\hat{r},z^\delta_k-z^\delta_j,y,\omega)\leq \delta C_\epsilon+C\epsilon-\eta.$$  We choose $\epsilon=\frac{\eta}{4C}$ to conclude that $z^\delta$ is a global subsolution of (\ref{cell}) for each $0<\delta<\frac{\eta}{4 C_\epsilon}$ and $\omega\in\Omega_1$.

For all $\omega\in\Omega_1$ and for all $\delta<\frac{\eta}{4 C_\epsilon}$, the comparison principle yields $z^\delta\leq v^\delta$ on $B_{R(\delta,\omega)}$.  In particular, for each $\eta>0$ and $\omega\in\Omega_1$, \begin{equation}\label{enhanced2} -\delta z^\delta_1(0,\omega)=\tilde{H}(p,r)+\eta\geq -\delta v^\delta_1(0,\omega).\end{equation}  In view of (\ref{enhance1}) and (\ref{enhanced2}), on a subset of full probability, \begin{equation}\label{enhance3} \tilde{H}(p,r)=\overline{H}(p,r)=\limsup_{\delta\rightarrow 0} -\delta v^\delta_1(0,\omega).\end{equation}

We now prove (\ref{charac1}).  Fix $(p,r)\in\mathbb{R}^n\times\mathbb{R}$.  In view of (\ref{enhance3}), a basic measure-theoretic lemma implies (See \cite{AS,LS2}), \begin{equation}\label{enhance4} \lim_{\delta\rightarrow 0}\mathbb{E}(\abs{\overline{H}(p,r)+\delta v^\delta_1(0,\omega)})=0. \end{equation}

It remains to prove that, for each $R>0$, \begin{equation*}\label{enhance5} \lim_{\delta\rightarrow 0}\sup_{y\in B_{R/\delta}}\mathbb{E}(\abs{\overline{H}(p,r)+\delta v^\delta_1(y,\omega)})=0.\end{equation*}  Fix $R>0$ and $\epsilon>0$.  Using the Vitali covering lemma, choose balls $\left\{B(y_1,\epsilon),\ldots,B(y_k,\epsilon)\right\}$ satisfying $k\leq C(\frac{R}{\epsilon})^n$ and $B_R\subset\cup_{i=1}^kB(y_i,\epsilon)$.  In view of Propositions \ref{cellsol} and \ref{subsolution} there exists $C>0$ satisfying, for each $\delta>0$, $$\begin{array}{lll} \sup_{y\in B_{R/\delta}}\abs{\overline{H}(p,r)+\delta v^\delta_1(y,\omega)} & \leq & \max_{1\leq i\leq k}(\abs{\overline{H}(p,r)+\delta v^\delta_1(y_i/\delta,\omega)})+C\epsilon \\ & = & \max_{1\leq i\leq k}(\abs{\overline{H}(p,r)+\delta v^\delta_1(0,\omega)+\delta w^\delta_1(y_i/\delta,\omega)})+C\epsilon.\end{array}$$  Since $\mathbb{P}(\Omega_1)=1$, it follows from (\ref{enhance4}) and Proposition \ref{subsolution} that $$\limsup_{\delta\rightarrow 0}\sup_{y\in B_{R/\delta}}\mathbb{E}(\abs{\overline{H}(p,r)+\delta v^\delta_1(y,\omega)})\leq \limsup_{\delta\rightarrow 0}\mathbb{E}(\abs{\overline{H}(p,r)+\delta v^\delta_1(0,\omega)})+C\epsilon=C\epsilon.$$  Because $\epsilon>0$ was arbitrary, this completes the proof.\end{proof}

\section{Properties of $\overline{H}(p,r)$}

We identify the properties of the effective Hamiltonian $\overline{H}(p,r)$, which yield the well-posedness of the scalar equation \begin{equation}\label{eff}\left\{\begin{array}{ll} u_t+\overline{H}(Du,u)=0 & \textrm{on}\;\;\;\mathbb{R}^n\times (0,\infty), \\ u=u_0 & \textrm{on}\;\;\;\mathbb{R}^n\times\left\{0\right\}. \end{array}\right.\end{equation}

The continuity properties of $\overline{H}(p,r)$ are inherited from (\ref{hamcon}).  We remark that, because the estimates contained in Proposition \ref{cellsol} depend on $(p,r)$, only local continuity estimates are obtained for $\overline{H}$.  Furthermore, the effective Hamiltonian inherits the minimal coercivity occurring for the $H_k$'s in (\ref{coercive}).  And, $\overline{H}(p,r)$ is monotone in the variable $r$ and convex in the variable $p$.  The monotonicity is inherited from (\ref{colmonstrict}) and the convexity from (\ref{gradconvex}).

\begin{prop}  Assume (\ref{steady}).  For each $R>0$ there exists $C=C(R)>0$ such that, for all $(p_i,r_i)\in B_R\times B_R$, \begin{equation}\label{effcon}\abs{\overline{H}(p_1,r_1)-\overline{H}(p_2,r_2)}\leq C\max_{1\leq k\leq m}((1+\abs{p_1}+\abs{p_2})^{\gamma_k-1}\abs{p_1-p_2}+\abs{r_1-r_2}).\end{equation}

For each $R>0$ and for the same constants occurring in (\ref{coercive}) we have, for all $(p,r)\in\mathbb{R}^n\times B_R$, \begin{equation}\label{effcoercive}\min_{1\leq k\leq m} \left(C_1\abs{p}^{\gamma_k}-C_3\right)\leq \overline{H}(p,r).\end{equation}

If $r_1<r_2$ then, for each $p\in\mathbb{R}^n$, \begin{equation}\label{effmonotone}\overline{H}(p,r_1)\leq \overline{H}(p,r_2).\end{equation}

For each $r\in\mathbb{R}$, \begin{equation}\label{effconvex} p\rightarrow\overline{H}(p,r)\;\;\textrm{is convex.}\end{equation}\end{prop}

\begin{proof}  Each property is obtained using a straightforward comparison argument.  Therefore, we prove only (\ref{effcon}).  Fix $R>0$.  Let $v^{i,\delta}$ be the solution of (\ref{cell}) corresponding to $(p_i,r_i)\in B_R\times B_R$.  In view of (\ref{hamcon}) and Proposition \ref{cellsol}, there exists $C=C(R)>0$ such that, for each $\omega\in\Omega$, \begin{equation*}\label{effcon2}\begin{array}{lll} z^\delta=(z^\delta_1,\ldots,z^\delta_m) & \textrm{with} & z^\delta_k=v^{1,\delta}_k-\delta^{-1}C\max_{1\leq i\leq m}((1+\abs{p_1}+\abs{p_2})^{\gamma_i-1}\abs{p_1-p_2}+\abs{r_1-r_2})\end{array}\end{equation*} is a subsolution of (\ref{cell}) corresponding to $(p_2,r_2)$ and, by the comparison principle, \begin{equation*}\label{effcon3}\begin{array}{ll} \max_{1\leq k\leq m}(\delta v^{1,\delta}_k-\delta v^{2,\delta}_k)\leq C\max_{1\leq i\leq m}((1+\abs{p_1}+\abs{p_2})^{\gamma_i-1}\abs{p_1-p_2}+\abs{r_1-r_2}) & \textrm{on}\;\;\;\mathbb{R}^n.\end{array} \end{equation*}  The opposite inequality is obtained by reversing the roles of $v^{1,\delta}$ and $v^{2,\delta}$.

We conclude using Proposition \ref{auxenhance}.  \end{proof}

The well-posedness of (\ref{eff}) is now an immediate consequence of (\ref{effcon}), (\ref{effcoercive}) and (\ref{effmonotone}) (See Crandall, Ishii and Lions \cite{CIL}).

\begin{prop}\label{effsol}  Assume (\ref{steady}).  For each $T>0$, $(\ref{eff})$ admits a unique solution $u\in\BUC(\mathbb{R}^n\times[0,T))$ and, if $u_0\in C^{0,1}(\mathbb{R}^n)$, then there exists $C>0$ satisfying $\norm{u_t}_{L^\infty(\mathbb{R}^n\times[0,\infty))}\leq C$ and $\norm{Du}_{L^\infty(\mathbb{R}^n\times[0,\infty))} \leq C$. \end{prop}

\section{The Metric System}

In this section we introduce, for $\mu\in\mathbb{R}$, $(p,r)\in\mathbb{R}^n\times\mathbb{R}$ and $\omega\in\Omega_1$, the so called ``metric system'' \begin{equation}\label{metric} \left\{\begin{array}{ll} -\tr(A_k(y,\omega)D^2m_{k,\mu})+H_k(p+Dm_{k,\mu},\hat{r},m_{k,\mu}-m_{j,\mu},y,\omega)=\mu & \textrm{on}\;\;\;\mathbb{R}^n\setminus D, \\ m_\mu(\cdot,\omega)=w(\cdot,\omega)-w(x,\omega) & \textrm{on}\;\;\;\partial D, \end{array}\right.\end{equation} for $D$ a closed bounded subset of $\mathbb{R}^n$, $x\in D$ and $w$ the subsolution constructed in Proposition \ref{subsolution} corresponding to $(p,r)$.  The scalar version of (\ref{metric}) was considered in \cite{AS} to prove the homogenization of viscous Hamilton-Jacobi equations in unbounded environments, and the first proofs of homogenization for scalar, first-order Hamilton-Jacobi equations in \cite{LS1,S} were based on the behavior of the first-order, time-dependent version of (\ref{metric}).

We first aim to prove that the metric system is well-posed for each $\mu>\overline{H}(p,r)$.  To obtain this result, we will use crucially the fact that, for each $\mu>\overline{H}(p,r)$, we have by Proposition \ref{subsolution} a strict subsolution of (\ref{metric}) which grows sublinearly at infinity.

\begin{prop}\label{colmetriccomp} Assume (\ref{steady}).  Let $u\in\USC(\mathbb{R}^n;\mathbb{R}^m)$ and $v\in\LSC(\mathbb{R}^n;\mathbb{R}^m)$ be respectively a subsolution and a supersolution of (\ref{metric}) for $\mu>\overline{H}(p,r)$ and $\omega\in\Omega_1$ satisfying \begin{equation}\label{colmetriccomp1} \begin{array}{lll} \min_{1\leq k\leq m} \liminf_{\abs{y}\rightarrow\infty}\frac{v_k(y)}{\abs{y}}\geq 0 & \textrm{and} & \max_{1\leq k\leq m}\limsup_{\abs{y}\rightarrow\infty}\frac{u_k(y)}{\abs{y}}<\infty \end{array}\end{equation} with $u\leq v$ on $\partial D$.  Then $u\leq v$ on $\mathbb{R}^n\setminus D$. \end{prop}

\begin{proof}  Fix $(p,r)\in\mathbb{R}^n\times\mathbb{R}$.  In view of the known comparison principles, it suffices to show that (See \cite{CIL,IK}),\begin{equation*}\label{mcomp1} \min_{1\leq k\leq m}\liminf_{\abs{y}\rightarrow\infty}\frac{v_k(y)-u_k(y)}{\abs{y}}\geq 0.\end{equation*}  For each $k\in\left\{1,\ldots,m\right\}$, consider the set $$\Lambda_k=\left\{\;0\leq \lambda\leq 1\;|\;\liminf_{\abs{y}\rightarrow\infty}\frac{v_k(y)-\lambda u_k(y)}{\abs{y}}\geq 0\;\right\}.$$  The goal is to show that, for each $k\in\left\{1,\ldots,m\right\}$, \begin{equation}\label{compgoal} \Lambda_k=[0,1].\end{equation}

Fix $k\in\left\{1,\ldots,m\right\}$.  By assumption $0\in\Lambda_k$.  A basic argument proves that $\Lambda_k=[0,\overline{\lambda}_k]$ for some $0\leq \overline{\lambda_k}\leq 1$ (See \cite{AS}).  If $\limsup_{\abs{y}\rightarrow\infty}{u_k(y)}/{\abs{y}}\leq 0$ then $\overline{\lambda}_k=1$.  We may therefore assume that, for some $j\in\left\{1,\ldots,m\right\}$, we have $\limsup_{\abs{y}\rightarrow\infty}{u_j(y)}/{\abs{y}}\geq 0$.

To prove (\ref{compgoal}) we argue by contradiction.  Assume that \begin{equation*}\label{mcomp2} \overline{\lambda}=\min_{1\leq k\leq m}\overline{\lambda}_k<1, \end{equation*} where we include the possibility $\overline{\lambda}=0$.

For each $R>0$, let $$\varphi_R(y)=R-(R^2+\abs{y}^2)^{1/2}.$$  It is immediate that there exists $C>0$ such that $\norm{D\varphi_R}_{L^\infty(\mathbb{R}^n)}\leq C$ and $\norm{D^2\varphi_R}_{L^\infty(\mathbb{R}^n)}\leq C$ uniformly for $R\geq 1$ and, furthermore, as $R\rightarrow\infty$ \begin{equation}\label{mcomp3}\begin{array}{ll} \varphi_R\rightarrow 0 & \textrm{locally uniformly on}\;\;\;\mathbb{R}^n.\end{array}\end{equation}

Let $w$ denote the subsolution constructed in Proposition \ref{subsolution} corresponding to $(p,r)$.  By subtracting a constant we may assume $w\leq 0$ on $D$.

Let \begin{equation*}\label{comp4} \tilde{u}=(1-\overline{\lambda})w+\overline{\lambda}u\end{equation*}  and observe that, in view of (\ref{gradconvex}) and (\ref{difconvex}), $\tilde{u}$ satisfies \begin{equation*}\label{mcomp5}\left\{\begin{array}{ll} -\tr(A_k(y,\omega)D^2\tilde{u}_k)+H_k(p+D\tilde{u}_k,\hat{r},\tilde{u}_k-\tilde{u}_j,y,\omega) \leq (1-\overline{\lambda})\overline{H}(p,r)+\overline{\lambda}\mu & \textrm{on}\;\;\;\mathbb{R}^n\setminus D, \\ \tilde{u}\leq v & \textrm{on}\;\;\;\partial D. \end{array}\right.\end{equation*} Notice that, because $\mu>\overline{H}(p,r)$, this implies $\tilde{u}$ is a strict subsolution of (\ref{metric}).

For $0<\epsilon<1$ and $a>\max_{1\leq k\leq m}\limsup_{\abs{y}\rightarrow\infty}\frac{u_k(y)}{\abs{y}}\geq 0$, let $$\tilde{u}_R(y)=(1-\epsilon)\tilde{u}+\epsilon(u+a\varphi_R).$$  It follows from (\ref{gradconvex}) and (\ref{difconvex}) that, for some $C>0$ independent of $R\geq 1$, $\tilde{u}_R$ satisfies the system $$\left\{\begin{array}{ll} -\tr(A_k(y,\omega)D^2\tilde{u}_{k,R})+H_k(p+D\tilde{u}_{k,R},\hat{r},\tilde{u}_{k,R}-\tilde{u}_{j,R},y,\omega) & \textrm{on}\;\;\;\mathbb{R}^n\setminus D, \\ \;\;\;\;\;\;\;\;\leq (1-\epsilon)(1-\overline{\lambda})\overline{H}(p,r)+(1-\epsilon)\overline{\lambda}\mu+\epsilon\mu+aC\epsilon & \\ \tilde{u}_R\leq v+\max_{y\in D}(a\epsilon\varphi_R(y)) & \textrm{on}\;\;\;\partial D.\end{array}\right.$$  Observe that, because $\overline{\lambda}<1$ and $\mu>\overline{H}(p,r)$, $\tilde{u}_R$ is a strict subsolution of (\ref{metric}) for all $\epsilon$ sufficiently small.

In view of the choice of $a>0$, for each $k\in\left\{1,\ldots,m\right\}$, \begin{equation*} \liminf_{\abs{y}\rightarrow\infty}\frac{v_k(y)-\tilde{u}_{k,R}(y)}{\abs{y}}\geq \liminf_{\abs{y}\rightarrow\infty}\frac{v_k(y)-(1-\epsilon)\overline{\lambda}u_k(y)}{\abs{y}}-\epsilon\liminf_{\abs{y}\rightarrow\infty}\frac{u_k(y)+a\varphi_R(y)}{\abs{y}}\geq 0.\end{equation*}  The comparison principle implies, for each $R\geq 1$, \begin{equation*}\label{mcomp7}\max_{1\leq k\leq m}\sup_{y\in\mathbb{R}^n\setminus D} \left(\tilde{u}_{k,R}-v_k\right)\leq \sup_{y\in\partial D} (\epsilon a\varphi_R)\end{equation*}  and, after letting $R\rightarrow\infty$, we have by (\ref{mcomp3}) that \begin{equation*}\label{mcomp8}\max_{1\leq k\leq m}\sup_{y\in\mathbb{R}^n\setminus D}\left((1-\epsilon)(1-\overline{\lambda})w_k+((1-\epsilon)\overline{\lambda}+\epsilon) u_k-v_k\right)\leq 0.\end{equation*}  The strict sub-linearity of $w$ at infinity yields \begin{equation*}\label{mcomp9}\min_{1\leq k\leq m}\liminf_{\abs{y}\rightarrow\infty}\frac{v_k(y)-((1-\epsilon)\overline{\lambda}+\epsilon)u_k(y)}{\abs{y}}\geq 0.\end{equation*} Therefore, for each $k\in\left\{1,\ldots,m\right\}$ and each $\epsilon>0$ sufficiently small, $0<\overline{\lambda}< ((1-\epsilon)\overline{\lambda}+\epsilon)\in\Lambda_k$, contradicting the definition of $\overline{\lambda}$.  \end{proof}

Proposition \ref{colmetriccomp} yields that, for $\mu>\overline{H}(p,r)$, a solution to (\ref{metric}) is unique provided it satisfies the required growth conditions at infinity.  The existence of such a solution follows from Perron's method.  For this, it is necessary to build an appropriate supersolution of (\ref{metric}).

The supersolutions available depend on the minimal coercivity of the Hamiltonians $H_k$, and are constructed in a manner similar to the analogous fact in \cite{AS}.  We therefore only sketch the argument.  Let \begin{equation*}\label{mincoercive} \gamma=\min_{1\leq k\leq m}\gamma_k,\end{equation*} define, for each $x\in\mathbb{R}^n$,  \begin{equation*}\label{metricsubset} D_\epsilon(x)=\left\{\begin{array}{ll} \left\{x\right\} & A_k=0\;\;\textrm{for each}\;\;k\in\left\{1,\ldots,m\right\}\;\;\textrm{or}\;\; \gamma>2, \\ \overline{B}_\epsilon(x) & A_k\neq 0\;\;\textrm{for some}\;\;k\in\left\{1,\ldots,m\right\}\;\;\textrm{and}\;\; \gamma\leq 2, \end{array}\right.\end{equation*} and consider the metric system \begin{equation}\label{colmetricpart} \left\{\begin{array}{ll} -\tr(A_k(y,\omega)D^2m_{k,\mu})+H_k(p+Dm_{k,\mu},\hat{r},m_{k,\mu}-m_{j,\mu},y,\omega)=\mu & \textrm{on}\;\;\;\mathbb{R}^n\setminus D_1(x), \\ m_\mu(y)=w_k(y,\omega)-w_k(x,\omega) & \textrm{on}\;\;\;\partial D_1(x). \end{array}\right. \end{equation}

\begin{prop}\label{colmetricexist}  Assume (\ref{steady}).  For each $(p,r)\in\mathbb{R}^n\times\mathbb{R}$, $\omega\in\Omega_1$ and $\mu>\overline{H}(p,r)$, (\ref{colmetricpart}) admits a unique solution $m_\mu=(m_{1,\mu},\ldots,m_{m,\mu})$ subject to the growth conditions \begin{equation}\label{colmetricgrowth} 0\leq \min_{1\leq k\leq m} \liminf_{\abs{y}\rightarrow\infty}\frac{m_{k,\mu}(y,\omega)}{\abs{y}}\leq \max_{1\leq k\leq m}\limsup_{\abs{y}\rightarrow\infty}\frac{m_{k,\mu}(y,\omega)}{\abs{y}}<\infty.\end{equation}  Furthermore, there exists $C=C(p,r)>0$ such that $$\max_{1\leq k\leq m}\norm{Dm_{\mu,k}}_{L^\infty(\mathbb{R}^n\setminus D_1(x))}\leq C.$$ \end{prop}

\begin{proof}  The uniqueness follows by Proposition (\ref{colmetriccomp}).  The regularity follows by Bernstein's method and is nothing more than a repetition of the argument presented in Proposition \ref{cellsol}.

The existence of a solution follows by Perron's method.  Basic properties of viscosity solutions imply that it suffices to consider $w=(w_1,\ldots,w_m)$ satisfying, for each $k\in\left\{1,\ldots,m\right\}$, $w_k\in\Lip(\mathbb{R}^n)$ and $Dw_k\in\Lip(\mathbb{R}^n;\mathbb{R}^n)$.

Proposition \ref{statsub} yields that $\tilde{w}(\cdot,\omega)=w(\cdot,\omega)-w(x,\omega)$ is a global subsolution of (\ref{metric}) satisfying $\tilde{w}(\cdot,\omega)=w(\cdot,\omega)-w(x,\omega)$ on $\partial D_1(x)$.  It remains to construct an appropriate supersolution.

In view of the definition of $\gamma$, there exists $a>0$ such that \begin{equation*}\label{metricsuper} z(y)=\left\{\begin{array}{ll} \tilde{w}(y)+a\abs{y-x} & \textrm{if}\;\;A_k=0\;\;\textrm{for all}\;\;k\in\left\{1,\ldots,m\right\}, \\ \tilde{w}(y)+a(\abs{y-x}^{(\gamma-2)/(\gamma-1)}+\abs{y-x}) & \textrm{if}\;\;A_k\neq 0\;\;\textrm{for some}\;\;k\;\;\textrm{and}\;\;\gamma>2, \\ \tilde{w}(y)+a(\abs{y-x}-1) & \textrm{if}\;\;A_k\neq 0\;\;\textrm{for some}\;\;k\;\;\textrm{and}\;\;\gamma\leq 2, \end{array} \right.\end{equation*} is a supersolution of (\ref{colmetricpart}) satisfying $z=\tilde{w}$ on $\partial D_1(x)$.

Perron's method yields a solution $m_\mu=(m_{1,\mu},\ldots,m_{m,\mu})$ of (\ref{colmetricpart}) satisfying $$\begin{array}{ll}\tilde{w}\leq m_\mu\leq z & \textrm{on}\;\;\;\mathbb{R}^n\setminus D_1(x)\end{array}$$  with $m_\mu=\tilde{w}$ on $\partial D_1(x)$, which implies (\ref{colmetricgrowth}).\end{proof}

For each $(p,r)\in\mathbb{R}^n\times\mathbb{R}$, we write $m_\mu(y,x,\omega)$ for the solution of (\ref{colmetricpart}) corresponding to $\mathbb{R}^n\setminus D_1(x)$ and $\omega\in\Omega_1$.  In the case $A_k\neq 0$ for some $k\in\left\{1,\ldots,m\right\}$ and $\gamma\leq 2$, we extend $m_\mu(y,x,\omega)$ to $\mathbb{R}^n\times\mathbb{R}^n\times\Omega_1$ by \begin{equation}\label{metext}\begin{array}{ll} m_\mu(y,x,\omega)=w(y,\omega)-w(x,\omega) & \textrm{for all}\;\;\;\abs{x-y}<1.\end{array}\end{equation}

\begin{prop}\label{metlip} Assume (\ref{steady}).  For each $\mu>\overline{H}(p,r)$ and $\omega\in\Omega_1$, there exists $C=C(p,r)>0$ such that $$\max_{1\leq k\leq m}\norm{Dm_{k,\mu}}_{L^{\infty}(\mathbb{R}^n\times\mathbb{R}^n)}\leq C.$$\end{prop}

\begin{proof}  In view of Proposition \ref{colmetricexist}, it suffices to prove that there exists $C=C(p,r)>0$ satisfying, for each $y\in\mathbb{R}^n$ and $\omega\in\Omega_1$, $$\max_{1\leq k\leq m}\norm{Dm_{k,\mu}(y,\cdot,\omega)}_{L^\infty(\mathbb{R}^n)}\leq C.$$  Fix $y\in\mathbb{R}^n$, $\omega\in\Omega_1$ and $\mu>\overline{H}(p,r)$ and let $x_1,x_2\in\mathbb{R}^n$.  The comparison principle implies $$\max_{1\leq k\leq m }(\sup_{y\in\mathbb{R}^n}m_{k,\mu}(y,x_1,\omega)-m_{k,\mu}(y,x_2,\omega))=\max_{1\leq k\leq m}(\max_{y\in D_1(x_1)\cup D_1(x_2)}m_{k,\mu}(y,x_1,\omega)-m_{k,\mu}(y,x_2,\omega)).$$

If $A_k=0$ for each $k\in\left\{1,\ldots,m\right\}$ or $\gamma>2$, we conclude in view of Proposition \ref{colmetricexist} that there exists $C=C(p,r)>0$ satisfying, for each $\mu>\overline{H}(p,r)$ and $\omega\in\Omega_1$, $$\max_{1\leq k\leq m}(\sup_{y\in\mathbb{R}^n}m_{k,\mu}(y,x_1,\omega)-m_{k,\mu}(y,x_2,\omega))\leq C\abs{x_1-x_2}.$$

If $A_k\neq 0$ for some $k\in\left\{1,\ldots,m\right\}$ and $\gamma\leq 2$, Propositions \ref{colmetricexist} and \ref{subsolution} yield that, for every $y\in D_1(x_1)$, there exists $C=C(p,r)>0$ satisfying, for each $k\in\left\{1,\ldots,m\right\}$, \begin{multline*} m_{k,\mu}(y,x_1,\omega)-m_{k,\mu}(y,x_2,\omega)=\left(w_k(y,\omega)-w_k(x_1,\omega)\right)-m_{k,\mu}(y,x_2,\omega) \\ \leq w_k(x_2,\omega)-w_k(x_1,\omega)\leq C\abs{x_1-x_2}.  \end{multline*}

If $y\in D_2(x)\setminus D_1(x_1)$, we have \begin{multline}\label{metlip1} \max_{1\leq k\leq m}\left(m_{k,\mu}(y,x_1,\omega)-m_{k,\mu}(y,x_2,\omega)\right)=\max_{1\leq k\leq m}\left(m_{k,\mu}(y,x_1,\omega)-(w_k(y,\omega)-w_k(x_2,\omega)\right) \\ =\max_{1\leq k\leq m}\left(m_{k,\mu}(y,x_1,\omega)-(w_k(y,\omega)-w_k(x_1,\omega))+(w_k(x_2,\omega)-w_k(x_1,\omega))\right).\end{multline} In view of Propositions \ref{subsolution} and \ref{colmetricexist}, there exists $C=C(p,r)>0$ satisfying, for each $k\in\left\{1,\ldots,m\right\}$, $$m_{k,\mu}(y,x_1,\omega)-(w_k(y,\omega)-w_k(x_1,\omega))\leq C(\abs{y-x_1}-1)\leq C\abs{x_2-x_1}$$ and $$\abs{w_k(x_2,\omega)-w_k(x_1,\omega)}\leq C\abs{x_2-x_1}.$$  Therefore, using (\ref{metlip1}), $$\max_{1\leq k\leq m}\left(m_{k,\mu}(y,x_1,\omega)-m_{k,\mu}(y,x_2,\omega)\right)\leq C\abs{x_2-x_1}.$$

We obtain the opposite inequality by reversing the roles of $x_1$ and $x_2$.  \end{proof}

We show next that the processes $m_\mu(x,y,\omega)$ are jointly stationary and subadditive up to a modification in the case $A_k\neq 0$ for some $k\in\left\{1,\ldots,m\right\}$ and $\gamma\leq 2$.  These facts, together with the subadditive ergodic theorem, will be used in the next section to prove the homogenization of (\ref{colmetricpart}).

\begin{prop}\label{colsub1}  Assume (\ref{steady}).  If $A_k=0$ for each $k\in\left\{1,\ldots,m\right\}$ or $\gamma>2$, then for all $(p,r)\in\mathbb{R}^n\times\mathbb{R}$, $\mu>\overline{H}(p,r)$ and $\omega\in\Omega_1$, the processes $m_\mu$ are jointly stationary in the sense that, for all $x,y,z\in\mathbb{R}^n$, \begin{equation}\label{colmetricjoint} m_\mu(y,x,\tau_z\omega)=m_\mu(y+z,x+z,\omega), \end{equation} and subadditive in the sense that, for all $x,y,z\in\mathbb{R}^n$, \begin{equation}\label{subadditive5} m_\mu(y,x,\omega)\leq m_\mu(y,z,\omega)+m_\mu(z,x,\omega). \end{equation}

If $A_k\neq 0$ for some $k\in\left\{1,\ldots,m\right\}$ and $\gamma\leq 2$, then there exists $C=C(p,r)>0$ such that, for each $(p,r)\in\mathbb{R}^n\times\mathbb{R}$, $\mu>\overline{H}(p,r)$ and $\omega\in\Omega_1$, the process $$\tilde{m}_\mu(y,x,\omega)=m_\mu(y,x,\omega)+C$$ is stationary and subadditive in the sense of (\ref{colmetricjoint}) and (\ref{subadditive5}). \end{prop}

\begin{proof}  We omit for both cases the proof of (\ref{colmetricjoint}), which is identical to the proof of Proposition \ref{cellsol} and follows immediately from Proposition \ref{colmetricexist} and (\ref{stationary1}).  To prove (\ref{subadditive5}), we first consider the case that either $A_k=0$ for each $k\in\left\{1,\ldots,m\right\}$ or $\gamma>2$ and fix $x,y,z\in\mathbb{R}^n$ and $\omega\in\Omega_1$.  It follows from Proposition \ref{colmetriccomp} that, for all $y\in\mathbb{R}^n$, $$w(y,\omega)-w(x,\omega)\leq m(y,x,\omega).$$  After reversing the roles of $x$ and $y$, we conclude that \begin{equation}\label{colmetricmiracle} 0\leq m_\mu(y,x,\omega)+m_\mu(x,y,\omega).\end{equation}

In view of Proposition \ref{colmetriccomp} with $D=\left\{x,z\right\}$, we have $$\begin{array}{ll} m_\mu(y,x,\omega)\leq m_\mu(y,z,\omega)+m_\mu(z,x,\omega) & \textrm{for all}\;\;\;y\in\mathbb{R}^n\end{array}$$ provided the inequality holds on $D=\left\{x,z\right\}$.  But, we have equality at $y=z$ and the inequality at $y=x$ is (\ref{colmetricmiracle}).

We now consider the case that $A_k\neq 0$ for some $k\in\left\{1,\ldots,m\right\}$ and $\gamma\leq 2$.  We fix $x,y,z\in\mathbb{R}^n$ and $\omega\in\Omega_1$ and choose $C>0$ satisfying $$C>2\max_{1\leq k\leq m}\norm{Dm_{k,\mu}}_{L^\infty(\mathbb{R}^n\times\mathbb{R}^n)}.$$

In view of Proposition \ref{colmetriccomp} with $D=D_1(x)\cup D_1(z)$, we have $$\begin{array}{ll} \tilde{m}_\mu(y,x,\omega)\leq\tilde{m}_\mu(y,z,\omega)+\tilde{m}_\mu(z,x,\omega) & \textrm{for all}\;\;\;y\in\mathbb{R}^n\end{array}$$ provided the inequality holds on $D_1(x)\cup D_1(z)$.

If $y\in D_1(x)$, then \begin{multline*} \tilde{m}_{k,\mu}(y,x,\omega)=w_k(y)-w_k(x)+C=(w_k(y,\omega)-w_k(z,\omega))+(w_k(z,\omega)-w_k(x,\omega))+C \\ \leq m_{k,\mu}(y,z,\omega)+m_{k,\mu}(z,x,\omega)+C\leq \tilde{m}_{k,\mu}(y,z,\omega)+\tilde{m}_{k,\mu}(z,x,\omega),\end{multline*} where Proposition \ref{colmetriccomp} is used to obtain the second to last inequality.

If $y\in D_1(z)$, then $\abs{m_{k,\mu}(y,z,\omega)}=\abs{w_k(y,\omega)-w_k(z,\omega)}\leq C/2$ and \begin{multline*}\tilde{m}_{k,\mu}(y,x,\omega)\leq m_{k,\mu}(z,x,\omega)+3C/2 \\ \leq m_{k,\mu}(z,x,\omega)+w_k(y,\omega)-w_k(z,\omega)+2C=\tilde{m}_{k,\mu}(y,z,\omega)+\tilde{m}_{k,\mu}(z,x,\omega).\end{multline*} \end{proof}

We conclude this section with a continuous dependence estimate, similar to Proposition \ref{contdep}, for the $m_\mu$ with respect to $(p,r)$.  As before, we only obtain local continuity estimates.

\begin{prop}\label{metcontdep}  Assume (\ref{steady}).  Fix $\mu>\overline{H}(p,r)$, $R>0$ and write $m^i_\mu$ for the solution of (\ref{colmetricpart}) corresponding to $(p_i,r_i)\in B_R\times B_R$.

If $A_k=0$ for each $k\in\left\{1,\ldots,m\right\}$ or $\gamma>2$ then, for each $\omega\in\Omega_1$, there exists $C=C(R)>0$ satisfying, for all $(p_i,r_i)\in B_R\times B_R$, $$\begin{array}{ll} \max_{1\leq k\leq m}\abs{m^1_{k,\mu}(y,x,\omega)-m^2_{k,\mu}(y,x,\omega)}\leq \frac{C}{\mu-\overline{H}(p,r)}(\abs{p_1-p_2}+\abs{r_1-r_2})\abs{y-x} & \textrm{on}\;\;\;\mathbb{R}^n\times\mathbb{R}^n.\end{array}$$

If $A_k\neq 0$ for some $k\in\left\{1,\ldots,m\right\}$ and $\gamma\leq 2$ then, for each $\omega\in\Omega_1$, there exists $C=C(R)>0$ satisfying, for all $(p_i,r_i)\in B_R\times B_R$, $$\begin{array}{ll} \max_{1\leq k\leq m}\abs{m^1_{k,\mu}(y,x,\omega)-m^2_{k,\mu}(y,x,\omega)}\leq C+\frac{C}{\mu-\overline{H}(p,r)}(\abs{p_1-p_2}+\abs{r_1-r_2})\abs{y-x} & \textrm{on}\;\;\;\mathbb{R}^n\times\mathbb{R}^n.\end{array}$$\end{prop}

\begin{proof}  Fix $\omega\in\Omega_1$ and $x\in\mathbb{R}^n$.  Let $w^i$ be the subsolution of (\ref{cellsub}) constructed in Proposition \ref{subsolution} corresponding to $(p_i,r_i)$.  And, let $\tilde{w}^i(\cdot)=w^i(\cdot)-w^i(x)$.

We first consider the case that $A_k=0$ for each $k\in\left\{1,\ldots,m\right\}$ or $\gamma>2$.  In view of Proposition \ref{metlip} and (\ref{hamcon}), there exists $C=C(R)>0$ such that $m^1_\mu$ is a solution of the inequality \begin{equation}\label{metcontdep1}-\tr(A_k(y,\omega)D^2m^1_{k,\mu})+H_k(p_2+Dm^1_{k,\mu},\hat{r}_2,m^1_{k,\mu}-m^1_{j,\mu},y,\omega)\leq \mu+C(\abs{p_1-p_2}+\abs{r_1-r_2})\end{equation} on $\mathbb{R}^n\setminus \left\{x\right\}$ satisfying $m^1_\mu(x)=0$.

Define $0<\lambda\leq1$ by $$\lambda=\frac{\mu-\overline{H}(p,r)}{C(\abs{p_1-p_2}+\abs{r_1-r_2})+(\mu-\overline{H}(p,r))}.$$  In view of (\ref{gradconvex}) and (\ref{difconvex}), the function $\lambda m^1_\mu+(1-\lambda)\tilde{w}^2$ is a subsolution of (\ref{colmetricpart}) corresponding to $(p_2,r_2)$.

Proposition \ref{colmetriccomp} implies $$\begin{array}{ll} m^1_\mu-m^2_\mu\leq m^1_\mu-(\lambda m^1_\mu+(1-\lambda)\tilde{w}^2)=(1-\lambda)(m^1_\mu-\tilde{w}^2) & \textrm{on}\;\;\;\mathbb{R}^n.\end{array}$$  In view of Propositions \ref{subsolution} and \ref{colmetricexist}, there exists $C=C(R)>0$ satisfying, for every $y\in\mathbb{R}^n$, $$\max_{1\leq k\leq m}(m^1_{k,\mu}(y,x,\omega)-\tilde{w}^2_k(y,\omega))\leq C\abs{y-x}.$$  Therefore, by the definition of $\lambda$, we conclude that there exists $C=C(R)>0$ such that, for each $y\in\mathbb{R}^n$, $$\max_{1\leq k\leq m}(m^1_{k,\mu}(y,x,\omega)-m^2_{k,\mu}(y,x,\omega))\leq \frac{C}{\mu-\overline{H}(p,r)}(\abs{p_1-p_2}+\abs{r_1-r_2})\abs{y-x}.$$  We obtain the opposite inequality by reversing the roles of $m^1_\mu$ and $m^2_\mu$.

We now consider the case that $A_k\neq 0$ for some $k\in\left\{1,\ldots,m\right\}$ and $\gamma\leq 2$.  In view of Proposition \ref{subsolution}, for $C_1=C_1(R)>0$, $$\max_{1\leq k\leq m}\norm{\tilde{w}^1_k-\tilde{w}^2_k}_{L^\infty(D_1(x))}\leq \max_{1\leq k\leq m}(\norm{D\tilde{w}^1_k}_{L^\infty(\mathbb{R}^n)}+\norm{D\tilde{w}^2_k}_{L^\infty(\mathbb{R}^n)})\leq C_1.$$  And, therefore, $m^1_\mu-\widehat{C}_1$ is a subsolution of (\ref{metcontdep1}) on $\mathbb{R}^n\setminus D_1(x)$ satisfying $m^1_\mu-\widehat{C}_1\leq \tilde{w}^2=m^2_\mu$ on $D_1(x)$.

For $0<\lambda\leq 1$ as above, the function $\lambda(m^1_\mu-\widehat{C}_1)+(1-\lambda)\tilde{w}^2$ is a subsolution of (\ref{colmetricpart}) corresponding to $(p_2,r_2)$.  This, in view of Proposition \ref{colmetriccomp}, implies $$\begin{array}{ll} m^1_\mu-m^2_\mu\leq m^1_\mu-(\lambda(m^1_\mu-\widehat{C}_1)+(1-\lambda)\tilde{w}^2)=\lambda\widehat{C}_1+(1-\lambda)(m^1_\mu-\tilde{w}^2) & \textrm{on}\;\;\mathbb{R}^n.\end{array}$$  It follows from Propositions \ref{subsolution} and \ref{colmetricexist} that there exists $C_2=C_2(R)>0$ satisfying $$\max_{1\leq k\leq m}\abs{m^1_{k,\mu}(y,x,\omega)-\tilde{w}^2_k(y,\omega)}\leq C_1+C_2\abs{y-x}.$$  Therefore, by definition of $\lambda$, there exists $C=C(R)>0$ satisfying, for each $y\in\mathbb{R}^n$, $$\max_{1\leq k\leq m}(m^1_{k,\mu}(y,x,\omega)-m^2_{k,\mu}(y,x,\omega))\leq C+\frac{C}{\mu-\overline{H}(p,r)}(\abs{p_1-p_2}+\abs{r_1-r_2})\abs{y-x}.$$  We obtain the opposite inequality be reversing the roles of $m^1_\mu$ and $m^2_\mu$. \end{proof}

\section{The Homogenization and Collapse of the Metric System}

We introduce, for each $(p,r)\in\mathbb{R}^n\times\mathbb{R}$, $\mu>\overline{H}(p,r)$, $\omega\in\Omega_1$ and $\epsilon>0$, the rescaled metric system \begin{equation}\label{colmetricres} \left\{\begin{array}{ll} -\epsilon\tr(A_k(\frac{y}{\epsilon},\omega)D^2 m^\epsilon_{k,\mu})+H_k(p+Dm^\epsilon_{k,\mu},\hat{r},\frac{m^\epsilon_{k,\mu}-m^\epsilon_{j,\mu}}{\epsilon},\frac{y}{\epsilon},\omega)=\mu & \textrm{on}\;\;\; \mathbb{R}^n\setminus D_\epsilon(0), \\ m^\epsilon_\mu(\cdot)=\epsilon w(\frac{\cdot}{\epsilon})-\epsilon w(0) & \textrm{on}\;\;\;\partial D_\epsilon(0), \end{array}\right. \end{equation} which, in view of Proposition \ref{colmetricexist}, has the unique solution $m^\epsilon_\mu(y,0,\omega)=\epsilon m_\mu(\frac{y}{\epsilon},0,\omega)$, for $m_\mu$ the solution of (\ref{colmetricpart}) corresponding to $(p,r)$.

The main result of this section is the homogenization and collapse of (\ref{colmetricres}) to a deterministic scalar equation.  More precisely, we prove that, as $\epsilon\rightarrow 0$, for each $k\in\left\{1,\ldots,m\right\}$, \begin{equation*}\label{methom1}\begin{array}{ll} m^\epsilon_{k,\mu}\rightarrow\overline{m}_\mu & \textrm{locally uniformly on}\;\;\;\mathbb{R}^n,\end{array}\end{equation*} for $\overline{m}_\mu:\mathbb{R}^n\rightarrow\mathbb{R}$ satisfying \begin{equation}\label{methom2}\left\{\begin{array}{ll} \overline{H}(p+D\overline{m}_\mu,r)=\mu & \textrm{on}\;\;\;\mathbb{R}^n\setminus\left\{0\right\}, \\ \overline{m}_\mu(0)=0.\end{array}\right.\end{equation}

We first make use of the subadditive ergodic theorem to identify the limit, as $\epsilon\rightarrow 0$, of the solutions $m^\epsilon_\mu$ almost surely.  We then show that this limit is deterministic and coincides with $\overline{m}_\mu$, the solution of (\ref{methom2}).

\begin{prop}\label{sublimit}  Assume (\ref{steady}).  There exists a subset $\Omega_2\subset\Omega_1$ of full probability such that, for each $(p,r)\in\mathbb{R}^n\times\mathbb{R}$, there exists $M_\mu\in \Lip(\mathbb{R}^n)$ satisfying, for each $y\in\mathbb{R}^n$, $\omega\in\Omega_2$ and $k\in\left\{1,\ldots,m\right\}$, $$\lim_{t\rightarrow \infty}\frac{1}{t}m_{k,\mu}(ty,0,\omega)=M_\mu(y).$$\end{prop}

\begin{proof}  Fix $(p,r)\in\mathbb{R}^n\times\mathbb{R}$.  We first show that there exists a subset $\Omega_2(p,r)$ satisfying the conclusion of Proposition \ref{sublimit} for this $(p,r)$.

Define for each $y\in\mathbb{R}^n$ the semigroup $\left(\sigma_t\right)_{t\geq0}$ of measure-preserving transformations of $\Omega$ by $$\sigma_t(\omega)=\tau_{ty}(\omega).$$  In view of Proposition \ref{colsub1}, in the case $A_k=0$ for each $k\in\left\{1,\ldots,m\right\}$ or $\gamma>2$, we have, for each $k\in\left\{1,\ldots,m\right\}$, $$m_{k,\mu}((t_1+t_2)y,0,\omega)\leq m_{k,\mu}(t_1y,0,\sigma_{t_2}\omega)+m_{k,\mu}(t_2y,0,\omega),$$  and, in the case $A_k\neq 0$ for some $k\in\left\{1,\ldots,m\right\}$ and $\gamma\leq 2$, with the notation of Proposition \ref{colsub1}, we have, for each $k\in\left\{1,\ldots,m\right\}$, $$\tilde{m}_{k,\mu}((t_1+t_2)y,0,\omega)\leq \tilde{m}_{k,\mu}(t_1y,0,\sigma_{t_2}\omega)+\tilde{m}_{k,\mu}(t_2y,0,\omega).$$

The subadditive ergodic theorem (See Proposition \ref{subadditiveergodic}) implies that for each $y\in\mathbb{R}^n$ there exists a subset $\Omega_y\subset\Omega$ of full probability and a random variable $M_\mu(y,\omega)=(M_{1,\mu}(y,\omega),\ldots,M_{m,\mu}(y,\omega))$ such that \begin{equation*}\label{sublimit1}\begin{array}{ll} \lim_{t\rightarrow\infty}\frac{1}{t}m_\mu(ty,0,\omega)=M_\mu(y,\omega) & \textrm{for each}\;\;\;\omega\in\Omega_y.\end{array}\end{equation*}

Let $\Omega_2(p,r)=\Omega_1\cap(\cap_{y\in\mathbb{Q}^n}\Omega_y).$  Then, in view of Proposition \ref{metlip}, we define, for every $\omega\in\Omega_2(p,r)$ and $y\in\mathbb{R}^n$, \begin{equation}\label{metlimit} M_\mu(y,\omega)=\lim_{t\rightarrow\infty}\frac{1}{t}m(ty,0,\omega).\end{equation}

We now show that $M_\mu$ is deterministic.  In view of the assumed ergodicity, it suffices to show that, for each $x,y\in\mathbb{R}^n$ and $\omega\in\Omega_2(p,r)$, $$M_\mu(y,\tau_x\omega)=M_\mu(y,\omega).$$  This again follows from Proposition \ref{metlip}.  Indeed, for each $x,y\in\mathbb{R}^n$ and $\omega\in\Omega_2(p,r)$, $$M_\mu(y,\tau_x\omega)=\lim_{t\rightarrow\infty}\frac{1}{t}m_\mu(ty,0,\tau_x\omega)=\lim_{t\rightarrow\infty}\frac{1}{t}m_\mu(ty+x,x,\omega)=M_\mu(y,\omega).$$

In view of Proposition \ref{metlip} and (\ref{metlimit}), it is immediate that $M_\mu\in \Lip(\mathbb{R}^n;\mathbb{R}^m)$.  Therefore, it remains only to show that $M_\mu$ is scalar.  By repeating the argument presented in Proposition \ref{cellsol} with $\epsilon^{-1}(m^\epsilon_{k,\mu}-m^\epsilon_{j,\mu})$ playing the role of $(v^\delta_k-v^\delta_j)$, for each $R>0$ there exists $C=C(R)>0$ satisfying $$\max_{1\leq i,j\leq m}\norm{m^\epsilon_{i,\mu}-m^\epsilon_{j,\mu}}_{L^\infty(B_R)}\leq \epsilon C $$ and, hence, for each $y\in\mathbb{R}^n$, $\omega\in\Omega_2(p,r)$ and $i,j\in\left\{1,\ldots,m\right\}$, $$ \lim_{\epsilon\rightarrow 0}m^\epsilon_{i,\mu}(y,0,\omega)=\lim_{t\rightarrow\infty}\frac{1}{t}m_{i,\mu}(ty,0,\omega)=\lim_{t\rightarrow\infty}\frac{1}{t}m_{j,\mu}(ty,0,\omega)=\lim_{\epsilon\rightarrow 0}m^\epsilon_{j,\mu}(y,0,\omega).$$  In view of (\ref{metlimit}), we conclude that $M_{i,\mu}=M_{j,\mu}$ for each $i,j\in\left\{1,\ldots,m\right\}$ and, henceforth, we will write $M_\mu$ for the scalar function $M_{k,\mu}$ for each $k\in\left\{1,\ldots,m\right\}$.

We conclude by defining $\Omega_2=\bigcap_{(p,r)\in\mathbb{Q}^n\times\mathbb{Q}}\Omega_2(p,r)$ and applying Proposition \ref{metcontdep}.  \end{proof}

In view of Proposition \ref{subsolution}, as $\epsilon\rightarrow 0$, the expected limiting equation for (\ref{colmetricres}) is \begin{equation}\label{colmetrichom} \left\{\begin{array}{ll} \overline{H}(p+\overline{m}_\mu,r)=\mu & \textrm{on}\;\;\;\mathbb{R}^n\setminus \left\{0\right\}, \\ \overline{m}_\mu(0)=0, & \end{array}\right. \end{equation} subject, in view of (\ref{colmetricgrowth}), to the condition \begin{equation}\label{colmetrichomgrowth} 0\leq \liminf_{\abs{y}\rightarrow\infty}\frac{\overline{m}_\mu(y)}{\abs{y}}\leq\limsup_{\abs{y}\rightarrow\infty}\frac{\overline{m}_\mu(y)}{\abs{y}}<\infty.\end{equation}

\begin{prop}\label{colmetricsol}  Assume (\ref{steady}).  For each $\mu>\overline{H}(p,r)$, $M_\mu$ is the unique solution of (\ref{colmetrichom}) subject to (\ref{colmetrichomgrowth}). \end{prop}

\begin{proof}  We begin with (\ref{colmetrichomgrowth}).  In view of (\ref{colmetricgrowth}), $M_\mu\geq0$ and, furthermore, it follows by definition that $M_\mu(y)$ is positively one-homogenous in the sense that, for each $t>0$ and $y\in\mathbb{R}^n$, we have $M_\mu(ty)=tM_\mu(y)$.  Therefore, $M_\mu$ satisfies (\ref{colmetrichomgrowth}) because $$0\leq \min_{\abs{y}=1}M_\mu(y)=\liminf_{\abs{y}\rightarrow\infty}\frac{M_\mu(y)}{\abs{y}}\leq \limsup_{\abs{y}\rightarrow\infty}\frac{M_\mu(y)}{\abs{y}}=\max_{\abs{y}=1}M_\mu(y)<\infty.$$

Next we prove $M_\mu$ is a subsolution of (\ref{colmetrichom}).  Arguing by contradiction, we assume that $M_\mu-\phi$ has a strict local maximum at $x_0\in\mathbb{R}^n\setminus\left\{0\right\}$ for $\phi\in C^2(\mathbb{R}^n)$ satisfying \begin{equation}\label{metrichomcontra}\overline{H}(p+D\phi(x_0),r)-\mu=\theta>0.\end{equation}

We follow now the classical perturbed test function method.  For each $\delta>0$, let $w^\delta$ be the solution of (\ref{cellnorm}) corresponding to $(p+D\phi(x_0),r)$.  We define the perturbed test function $$\begin{array}{ll} \phi^\epsilon=(\phi^\epsilon_1,\ldots,\phi^\epsilon_m) & \textrm{with}\;\;\;\phi^\epsilon_k(y)=\phi(y)+\epsilon w^\epsilon_k(\frac{y}{\epsilon},\omega).\end{array}$$

In view of Proposition \ref{auxenhance}, for almost every $\omega\in\Omega_2$, there exists a sequence $\epsilon_j=\epsilon_j(\omega)\rightarrow 0$ satisfying \begin{equation}\label{methom5} \lim_{j\rightarrow\infty}\abs{\epsilon_j v^{\epsilon_j}_1(0,\omega)+\overline{H}(p+D\phi(x_0),r)}=0.\end{equation}  We will show that, for each $\omega\in\Omega_2$ satisfying (\ref{methom5}) and for all $\epsilon_j$ sufficiently small, $\phi^{\epsilon_j}$ is a supersolution of (\ref{colmetricres}) near $x_0$.

We fix $\omega\in\Omega_2$ satisfying (\ref{methom5}) and suppress the dependence on $\omega$ in what follows to simplify notation.  Suppose that, for some $k\in\left\{1,\ldots,m\right\}$ and $\psi\in C^2(\mathbb{R}^n)$, the function $\phi^{\epsilon_j}_k-\psi$ has a local minimum at $y_0\in\mathbb{R}^n$.  Then, the rescaled function $$y\rightarrow w^{\epsilon_j}_k(y)-\frac{1}{\epsilon_j}\left(\psi(\epsilon_j y)-\phi(\epsilon_j y)\right)$$ achieves a local minimum at $\frac{y_0}{\epsilon_j}$.  Because $w^{\epsilon_j}$ is a solution of (\ref{cellnorm}), after returning to the original scaling and evaluated at $y_0$, \begin{equation*} \epsilon_j w^{\epsilon_j}_k-\epsilon_j\tr(A_k(\frac{y_0}{\epsilon_j})(D^2\psi-D^2\phi)) +H_k(p+D\phi(x_0)+D\psi-D\phi,\hat{r},\frac{\phi^{\epsilon_j}_k-\phi^{\epsilon_j}_i}{\epsilon},\frac{y_0}{\epsilon})\geq-\epsilon_j v^{\epsilon_j}_1(0,\omega).\end{equation*}

In view of (\ref{metrichomcontra}) and (\ref{methom5}) and because $\omega\in\Omega_2\subset\Omega_1$, there exists $j_0=j_0(\omega)$ such that, for all $j>j_0$, \begin{equation}\label{methom6}\begin{array}{lll} \abs{\epsilon_j w^{\epsilon_j}(\frac{y_0}{\epsilon_j},\omega)}<\frac{\theta}{4} & \textrm{and} & -\epsilon_j v^{\epsilon_j}_1(0,\omega)>\frac{3\theta}{4}+\mu.\end{array}\end{equation}  Using (\ref{hamcon}), Proposition \ref{cellnormsol} and because $\phi\in C^2(\mathbb{R}^n)$, there exists $r>0$ such that if $\abs{x_0-y_0}<r$ then, for all $j$ sufficiently large, \begin{equation*} \epsilon_j w^{\epsilon_j}_k-\epsilon_j\tr(A_k(\frac{y_0}{\epsilon_j},\omega)D^2\psi) +H_k(p+D\psi,\hat{r},\frac{\phi^{\epsilon_j}_k-\phi^{\epsilon_j}_i}{\epsilon},\frac{y_0}{\epsilon},\omega)\geq -\epsilon_j v^{\epsilon_j}_1(0,\omega)-\frac{\theta}{4}.\end{equation*}  We conclude by (\ref{methom6}) that if $\abs{x_0-y_0}<r$ then, for all $j>j_0$ sufficiently large, \begin{equation*} -\epsilon_j\tr(A_k(\frac{y_0}{\epsilon_j},\omega)D^2\psi) +H_k(p+D\psi,\hat{r},\frac{\phi^{\epsilon_j}_k-\phi^{\epsilon_j}_i}{\epsilon},\frac{y_0}{\epsilon},\omega)>\mu+\frac{\theta}{4}\end{equation*} and, therefore, that $\phi^{\epsilon_j}$ is a strict supersolution of (\ref{colmetrichom}) on $B_r(x_0)$.

The comparison principle implies, for all $j>j_0$ sufficiently large, $$\max_{1\leq k\leq m}\max_{x\in \overline{B}_r(x_0)}(m^{\epsilon_j}_{k,\mu}-\phi^{\epsilon_j}_k)=\max_{1\leq k\leq m}\max_{x\in\partial \overline{B}_r(x_0)}(m^{\epsilon_j}_{k,\mu}-\phi^{\epsilon_j}_k).$$  Since, as $j\rightarrow\infty$, for each $\omega\in\Omega_2$ and $k\in\left\{1,\ldots,m\right\}$, $$\begin{array}{ll} \phi^{\epsilon_j}_k\rightarrow\phi & \textrm{locally uniformly on}\;\;\;\mathbb{R}^n\end{array}$$ and $$\begin{array}{ll} m^{\epsilon_j}_{k,\mu}(\cdot,0,\omega)\rightarrow M_\mu(\cdot) &  \textrm{locally uniformly on}\;\;\;\mathbb{R}^n, \end{array}$$ standard optimization results yield a contradiction to the assumption that $M_\mu-\phi$ has a strict local maximum at $x_0$.

The proof that $M_\mu$ is a supersolution of (\ref{colmetrichom}) is analogous.  \end{proof}

The deterministic limit $M_\mu$ constructed in Proposition \ref{sublimit} is therefore the solution $\overline{m}_\mu$ of (\ref{methom2}) corresponding to $(p,r)$ and $\mu>\overline{H}(p,r)$.  We conclude this section by presenting some properties of the solution $\overline{m}_\mu$ which will be used in the sequel.  Because the arguments are scalar and identical to the analogous facts in \cite{AS} we omit the proofs.

\begin{prop}\label{colmetpositive}  Assume (\ref{steady}).  For each $(p,r)\in\mathbb{R}^n\times\mathbb{R}$ and each $\mu>\overline{H}(p,r)$, $\overline{m}_\mu$ is positively one-homogenous satisfying $\overline{m}_\mu>0$ on $\mathbb{R}^n\setminus\left\{0\right\}$. \end{prop}

We now obtain a formula for the solution $\overline{m}_\mu$ in terms of the effective Hamiltonian.

\begin{prop}  Assume (\ref{steady}).  For each $(p,r)\in\mathbb{R}^n\times\mathbb{R}$ and $\mu>\overline{H}(p,r)$, $\overline{m}_\mu$ is convex and given by the formula \begin{equation}\label{colmetricformula} \overline{m}_\mu(y)=\sup\left\{\;q\cdot y\;|\;q\in\mathbb{R}^n\;\;\textrm{satisfies}\;\;\overline{H}(p+q)\leq \mu\;\right\}.\end{equation}\end{prop}

We conclude this section with a characterization of $\overline{m}_\mu$ to be used in the sequel.

\begin{prop}\label{techlemma}  Assume (\ref{steady}).  Fix $(p,r)\in\mathbb{R}^n\times\mathbb{R}$ and $\mu>\overline{H}(p,r)$.  Suppose that $q_0\in\mathbb{R}^n$ satisfies $\overline{H}(p+q_0,r)=\mu$.  Then, there exists $x_0\in\mathbb{R}^n$ such that $\abs{x_0}=1$ and $\overline{m}_\mu(x_0)=x_0\cdot q_0$. \end{prop}

\section{A Further Identification of $\overline{H}(p,r)$}

We prove here that, almost surely in $\Omega$, for each $(p,r)\in\mathbb{R}^n\times\mathbb{R}$ and $R>0$, \begin{equation}\label{enhance2} \lim_{\delta\rightarrow 0}\sup_{y\in B_{R/\delta}}\abs{\overline{H}(p,r)+\delta v^\delta_1(y,\omega)}=0.\end{equation}  This convergence is necessary to apply the perturbed test function method in the following section.

In the following proposition, we provide a characterization of the minimum, for each $r\in\mathbb{R},$ of the function \begin{equation*}\begin{array}{ll} p\rightarrow\overline{H}(p,r) & \textrm{on}\;\;\;\mathbb{R}^n.\end{array}\end{equation*}  To this end, we consider the collection of $v\in\Lip(\mathbb{R}^n,\mathbb{R}^m)$ satisfying, for fixed $\omega\in\Omega$ and $\mu\in\mathbb{R},$ the system \begin{equation}\label{colminsystem}\begin{array}{ll} -\tr(A_k(y,\omega)D^2v_k)+H_k(p+Dv_k,\hat{r},v_k-v_j,y,\omega)\leq\mu & \textrm{on}\;\;\;\mathbb{R}^n.\end{array}\end{equation}

\begin{prop}\label{colmin}  Assume (\ref{steady}).  Fix $r\in\mathbb{R}$.  If $p\in\mathbb{R}^n$ satisfies $\overline{H}(p,r)=\min_{q\in\mathbb{R}^n}\overline{H}(q,r)$ then, on a subset of full probability, \begin{equation}\label{minformula} \overline{H}(p,r)=\inf\left\{\;\mu\;|\;\textrm{There exists}\;\;v\in\Lip(\mathbb{R}^n;\mathbb{R}^m)\;\;\textrm{satisfying}\;\;(\ref{colminsystem})\;\right\}.\end{equation} \end{prop}

\begin{proof}  Let $\tilde{H}(p,r,\omega)$ denote the righthand side of (\ref{minformula}).  The stationarity of the coefficients imply that, for each $x\in\mathbb{R}^n$, $\tilde{H}(p,r,\tau_x\omega)=\tilde{H}(p,r,\omega)$.  It follows from (\ref{transgroup}) that $\tilde{H}(p,r,\omega)=\tilde{H}(p,r)$ is deterministic.

For every $\omega\in\Omega_1$, the subsolution $w$ constructed in Proposition \ref{subsolution} is Lipschitz continuous and a subsolution of (\ref{colminsystem}) corresponding to $\mu=\overline{H}(p,r)$.  Therefore $\tilde{H}(p,r)\leq \overline{H}(p,r)$.  We now obtain the opposite inequality.

Let $\mu>\tilde{H}(p,r)$ and $\omega\in\Omega_1$.  By definition there exists $v=(v_1,\ldots,v_m)\in \Lip(\mathbb{R}^n;\mathbb{R}^m)$ satisfying (\ref{colminsystem}) where, by standard properties of viscosity solutions, we may assume that, for each $k\in\left\{1,\ldots,m\right\}$, $Dv_k\in\Lip(\mathbb{R}^n;\mathbb{R}^n)$.

We follow the proof of Proposition \ref{colmetricexist}.  Define $\tilde{v}(y)=v(y)-v(0)$ and observe that there exists $a>0$ such that $$z(y)=\left\{\begin{array}{ll} \tilde{v}+a\abs{y} & \textrm{if}\;\;A_k=0\;\;\textrm{for all}\;\;1\leq k\leq m, \\ \tilde{v}+a(\abs{y}^{(\gamma-2)/(\gamma-1)}+\abs{y}) & \textrm{if}\;\;A_k\neq 0\;\;\textrm{for some}\;\;k\;\;\textrm{and}\;\;\gamma>2, \\ \tilde{v}+a(\abs{y}-1) & \textrm{if}\;\;A_k\neq 0\;\;\textrm{for some}\;\;k\;\;\textrm{and}\;\;\gamma\leq 2, \end{array} \right.$$ is a supersolution of (\ref{colminsystem}) on $\mathbb{R}^n\setminus D_1(0)$ satisfying $z=\tilde{v}$ on $\partial D_1(0)$.

Perron's method yields a solution $s=(s_1,\ldots,s_m)$ of the system \begin{equation}\label{colminmetricsol} \left\{\begin{array}{ll} -\tr(A_k(y,\omega)D^2s_k)+H_k(p+Ds_k,\hat{r},s_k-s_j,y,\omega)=\mu & \textrm{on}\;\;\;\mathbb{R}^n\setminus D_1(0), \\ s=\tilde{v} & \textrm{on}\;\;\;\partial D_1(0),\end{array} \right.\end{equation} satisfying $\tilde{v}\leq s\leq z$ on $\mathbb{R}^n\setminus D_1(0)$.   Moreover, by repeating the proof of Proposition \ref{cellsol}, we have $s\in\Lip(\mathbb{R}^n;\mathbb{R}^m)$.

Observe that $s^\epsilon(\cdot)=\epsilon s(\frac{\cdot}{\epsilon})$ satisfies the rescaled system \begin{equation*}\label{colminmetricsol} \left\{\begin{array}{ll} -\epsilon\tr(A_k(\frac{y}{\epsilon},\omega)D^2 s^\epsilon_k)+H_k(p+D s^\epsilon_k,\hat{r}, \frac{s^\epsilon_k-s^\epsilon_j}{\epsilon},\frac{y}{\epsilon},\omega)=\mu & \textrm{on}\;\;\;\mathbb{R}^n\setminus D_\epsilon(0), \\ s^\epsilon=\tilde{v}^\epsilon & \textrm{on}\;\;\;\partial D_\epsilon(0),\end{array} \right.\end{equation*} with $\tilde{v}^\epsilon(\cdot)=\epsilon \tilde{v}(\frac{\cdot}{\epsilon})$.  And, by repeating the proof of Proposition \ref{cellsol} with $\epsilon^{-1}(s^\epsilon_k-s^\epsilon_j)$ playing the role of $(v^\delta_k-v^\delta_j)$, for each $R>0$ there exists $C=C(R)>0$ satisfying $$\max_{1\leq j,k\leq m}\norm{s^\epsilon_k-s^\epsilon_j}_{L^\infty(B_R)}\leq \epsilon C.$$

Define $$S(x)=\limsup_{\epsilon\rightarrow 0}\max_{1\leq k\leq m}\left\{\;s_k^\epsilon(y)\;|\;\abs{y-x}<\epsilon\;\right\}$$ and observe that $S\in \Lip(\mathbb{R}^n)$.  We remark that the condition $\mu>\overline{H}(p,r)$ in Proposition \ref{colmetricsol} is only used to ensure the existence of the $m^\epsilon_\mu$'s.  Therefore, since we have the $s^\epsilon$'s a priori and since $\omega\in\Omega_1$, by repeating the proof of Proposition \ref{colmetricsol} and using standard optimization results (See \cite{CIL}), we conclude that $S$ satisfies the effective equation $$\left\{\begin{array}{ll} \overline{H}(p+DS,r)\leq\mu & \textrm{on}\;\;\;\mathbb{R}^n\setminus\left\{0\right\}, \\ S(0)=0, & \end{array}\right.$$  and, therefore, that $\overline{H}(p,r)=\min_{q\in\mathbb{R}^n}\overline{H}(q,r)\leq \mu$ for each $\mu>\tilde{H}(p,r)$. \end{proof}

The homogenization of the metric problem, Proposition \ref{colmetricsol}, and the characterization of $\min_{p}\overline{H}(p_,r)$ are now used to prove (\ref{enhance2}).  We follow closely the methods of \cite{AS} and use Proposition \ref{techlemma} and $\overline{m}_\mu$ to construct, in the proof's final step, a supersolution of (\ref{colmetricres}) corresponding to $\overline{H}(p,r)$.

\begin{prop}\label{colupgrade}  Assume (\ref{steady}).  There exists a subset $\Omega_3\subset\Omega_2$ of full probability such that, for each $R>0$, $(p,r)\in\mathbb{R}^n\times\mathbb{R}$, and $\omega\in\Omega_3$, $$\lim_{\delta\rightarrow 0}\sup_{y\in B_{R/\delta}}\abs{\delta v ^\delta_1(y,\omega)+\overline{H}(p,r)}=0.$$\end{prop}

\begin{proof}  Fix $(p,r)\in\mathbb{R}^n\times\mathbb{R}$.  It is necessary to show that there exists $\Omega_3(p,r)\subset\Omega$ of full probability satisfying, for each $\omega\in\Omega_3(p,r)$ and $R>0$, \begin{equation*}\label{zeroconverge}\limsup_{\delta\rightarrow 0}(\sup_{y\in B_{R/\delta}}-\delta v^\delta_1(y,\omega))=\overline{H}(p,r)=\liminf_{\delta\rightarrow 0}(\inf_{y\in B_{R/\delta}}-\delta v^\delta_1(y,\omega)).\end{equation*}

Let $\Omega_3'(p,r)\subset\Omega$ be the subset of full probability satisfying, for every $\omega\in\Omega_3'(p,r)$, $$\begin{array}{lll}\hat{H}(p,r):=\liminf_{\delta\rightarrow 0}(-\delta v^\delta_1(0,\omega)) & \textrm{and} & \overline{H}(p,r):=\limsup_{\delta\rightarrow 0}(-\delta v^\delta_1(0,\omega)).\end{array}$$  We first prove that there exists $\tilde{\Omega}_3(p,r)\subset\Omega$ of full probability satisfying, for every $\omega\in\tilde{\Omega}_3(p,r)$ and $R>0$, \begin{equation}\label{zeroconverge1} \limsup_{\delta\rightarrow 0}(\sup_{y\in B_{R/\delta}}-\delta v^\delta_1(y,\omega))=\limsup_{\delta\rightarrow 0}(-\delta v^\delta_1(0,\omega))=\overline{H}(p,r)\end{equation} and, \begin{equation}\label{zeroconverge11} \liminf_{\delta\rightarrow 0}(\inf_{y\in B_{R/\delta}}-\delta v^\delta_1(y,\omega))=\liminf_{\delta\rightarrow 0}(-\delta v^\delta_1(0,\omega))=\hat{H}(p,r).\end{equation}

Using Egoroff's theorem and Proposition \ref{auxenhance}, for each $\rho>0$ there exists $\overline{\delta}(\rho)>0$ and a subset $E_\rho$ satisfying $\mathbb{P}(E_\rho)>1-\rho$ with \begin{equation}\begin{array}{ll}\label{conenh1} \sup_{\omega\in E_\rho} \left(-\delta v^\delta_1(0,\omega)-\overline{H}(p,r)\right)<\rho & \textrm{for all}\;\;\; 0<\delta<\overline{\delta}(\rho),\end{array}\end{equation} and using the ergodic theorem (See Proposition \ref{ergodic1}), for each $\rho>0$ there exists $F_\rho\subset\Omega$ of full probability satisfying, for each $\omega\in F_\rho$, \begin{equation}\label{colupgradeergodic} \lim_{R\rightarrow\infty}\dashint_{B_R}1_{E_\rho}(\tau_y\omega)dy=\mathbb{P}(E_\rho)\geq(1-\rho).\end{equation}

Define $F_0=\bigcap_{j=1}^\infty F_{2^{-j}}$ and fix $\omega\in F_0$, $R>0$ and $\rho=2^{-j}$ for some $j\in\mathbb{N}$.  Using (\ref{colupgradeergodic}), there exists $\overline{\delta}=\overline{\delta}(R,\rho,\omega)>0$ such that, for all $0<\delta<\overline{\delta}$, \begin{equation}\label{colupgrademeasure}\abs{\left\{\;y\in B_{R/\delta}\;|\;\tau_y\omega\in E_\rho\;\right\}}\geq(1-2\rho)\abs{B_{R/\delta}}.\end{equation}  This implies that for each $z\in B_{R/\delta}$ there exists $y\in B_{R/\delta}$ satisfying $\abs{z-y}\leq C\rho^{\frac{1}{n}} R\delta^{-1}$ and $\tau_y\omega\in E_\rho$.

Therefore, using Proposition \ref{cellsol} and (\ref{conenh1}), for each $z\in B_{R/\delta}$, for $y\in B_{R/\delta}$ as above, \begin{multline*} -\delta v^\delta_1(z,\omega)-\overline{H}(p,r)\leq \abs{\delta v^\delta_1(y,\omega)-\delta v^\delta_1(z,\omega)}-\delta v^\delta_1(y,\omega)-\overline{H}(p,r) \\ \leq C\rho^\frac{1}{n} R -\delta v^\delta_1(0,\tau_y\omega)-\overline{H}(p,r)\leq C\rho^\frac{1}{n} R+\rho. \end{multline*}  Since $\rho>0$ was arbitrary, for each $\omega\in F_0$ and $R>0$, $$\limsup_{\delta\rightarrow 0}(\sup_{y\in B_{R/\delta}}-\delta v^\delta_1(y,\omega))\leq \overline{H}(p,r).$$

We conclude that, for each $\omega\in F_0\cap\Omega'_3(p,r)$ and $R>0$, $$\limsup_{\delta\rightarrow 0}(\sup_{y\in B_{R/\delta}}-\delta v^\delta_1(y,\omega))=\limsup_{\delta\rightarrow 0}-\delta v^\delta_1(0,\omega)=\overline{H}(p,r).$$  An analogous argument proves that there exists a subset $\tilde{F}_0$ of full probability such that, for each $\omega\in\tilde{F}_0\cap\Omega'_3(p,r)$ and $R>0$, $$\liminf_{\delta\rightarrow 0}(\inf_{y\in B_{R/\delta}}-\delta v^\delta_1(y,\omega))=\liminf_{\delta\rightarrow 0}-\delta v^\delta_1(0,\omega)=\hat{H}(p,r).$$  Define $\tilde{\Omega}_3(p,r)=F_0\cap\tilde{F}_0\cap\Omega'_3(p,r)$ to conclude that (\ref{zeroconverge1}) and (\ref{zeroconverge11}) hold for every $\omega\in\tilde{\Omega}_3(p,r)$ and $R>0$.

To conclude, we show that there exists a subset $\Omega_3(p,r)\subset\tilde{\Omega}_3(p,r)\cap\Omega_2$ of full probability satisfying, for every $\omega\in\Omega_3(p,r)$, $$\hat{H}(p,r)=\liminf_{\delta\rightarrow 0}(-\delta v^\delta_1(0,\omega))=\limsup_{\delta\rightarrow 0}(-\delta v^\delta_1(0,\omega))=\overline{H}(p,r)$$ by considering two cases:  the case $\overline{H}(p,r)=\min_{q\in\mathbb{R}^n}\overline{H}(q,r)$ and the case $\overline{H}(p,r)>\min_{q\in\mathbb{R}^n}\overline{H}(q,r).$

In the first case, suppose that $\overline{H}(p,r)=\min_{q\in\mathbb{R}^n}\overline{H}(q,r)$.  Let $\Omega_3''(p,r)$ denote the subset of full probability satisfying (\ref{minformula}) and define $\Omega_3(p,r)=\Omega_3''(p,r)\cap\tilde{\Omega}_3(p,r)\cap\Omega_2$.

Fix $\omega\in\Omega_3(p,r)$ and choose a sequence $\delta_j=\delta_j(\omega)\rightarrow 0$ satisfying $$\lim_{j\rightarrow\infty}-\delta_j v^{\delta_j}_1(0,\omega)=\hat{H}(p,r).$$

In view of Proposition \ref{cellnormsol} and standard properties of viscosity solutions, after passing to a further subsequence $\delta_{j_k}=\delta_{j_k}(\omega)\rightarrow 0$, as $k\rightarrow\infty$, $$\begin{array}{ll} w^{\delta_{j_k}}\rightarrow w & \textrm{locally uniformly on}\;\;\; \mathbb{R}^n, \end{array}$$ for $w\in \Lip(\mathbb{R}^n;\mathbb{R}^m)$ satisfying \begin{equation*}\begin{array}{ll} -\tr(A_k(y,\omega)D^2w_k)+H_k(p+Dw_k,\hat{r},w_k-w_j,y,\omega)=\hat{H}(p,r) & \textrm{on}\;\;\;\mathbb{R}^n. \end{array}\end{equation*} Using Proposition \ref{colmin}, this implies that $\hat{H}(p,r)\geq\overline{H}(p,r)$ and, hence, by definition that $\hat{H}(p,r)=\overline{H}(p,r)$.

In the second case, suppose that $\overline{H}(p,r)>\min_{q\in\mathbb{R}^n}\overline{H}(q,r)$.  Define $\Omega_3(p,r)=\tilde{\Omega}_3(p,r)\cap\Omega_2$ and fix $\omega\in\Omega_3(p,r)$.  Using (\ref{effcoercive}), we assume without loss of generality that \begin{equation*}\label{minzero} \overline{H}(0,r)=\min_{q\in\mathbb{R}^n}\overline{H}(q,r).\end{equation*}  We proceed by contradiction.  Assume \begin{equation*}\label{rho} \rho:=\overline{H}(p,r)-\hat{H}(p,r)>0\end{equation*} and choose a sequence $\delta_j=\delta_j(\omega)\rightarrow 0$ satisfying $$\lim_{j\rightarrow\infty}-\delta v^\delta_1(0,\omega)=\hat{H}(p,r)=\overline{H}(p,r)-\rho.$$

Let $m^\epsilon_\mu$ be the solution of (\ref{colmetricres}) corresponding to $(0,r)$ and $\mu=\overline{H}(p,r)$.  Since $\overline{H}(p,r)>\overline{H}(0,r)$, Proposition \ref{techlemma} implies that there exists $\abs{x_0}=1$ satisfying $$\overline{m}_\mu(x_0)=x_0\cdot p.$$

Define for $\eta>0$, $v^\delta$ the solution of (\ref{cell}) corresponding to $(p,r)$ and $x_0$ as above, \begin{equation*}\label{minzero2}\begin{array}{ll} z^j=(z^j_1,\ldots,z^j_m) & \textrm{with}\;\;\;z^j_k(x)=x\cdot p+\delta_jv^{\delta_j}_k(\frac{x}{\delta_j},\omega)-\eta\abs{x-x_0}^2.\end{array}\end{equation*}  Since $\omega \in\Omega_3(p,r)$, for every $0<r<1$ and $\eta>0$ sufficiently small and $j$ sufficiently large, $z^j$ is a subsolution of the system \begin{equation*}\label{minzero3}\begin{array}{ll} -\delta_j\tr(A_k(\frac{y}{\delta_j},\omega)D^2z^j_k)+H(Dz^j_k,\hat{r},\frac{z^j_k-z^j_i}{\delta_j},\frac{y}{\delta_j},\omega)\leq \overline{H}(p,r)-\frac{\rho}{2} & \textrm{on}\;\;\;B_r(x_0).\end{array}\end{equation*}

The comparison principle implies, for each $j$ sufficiently large, \begin{equation}\label{minzero4} \max_{1\leq k\leq m}\max_{x\in \overline{B}_r(x_0)}(z^j_k(x)-m^{\delta_j}_{k,\mu}(x,0,\omega))=\max_{1\leq k\leq m}\max_{x\in\partial \overline{B}_r(x_0)}(z^j_k(x)-m^{\delta_j}_{k,\mu}(x,0,\omega)).\end{equation}

We write, for each $j\in\mathbb{N}$ and $k\in\left\{1,\ldots,m\right\}$, \begin{equation}\label{minzero5} z^j_k(x)-m^{\delta_j}_{k,\mu}(x,0,\omega)=(\delta_j v^{\delta_j}_k(\frac{x}{\delta_j},\omega)-\eta\abs{x-x_0}^2)+ (p\cdot x-\overline{m}_\mu(x))+ (\overline{m}_\mu(x)-m^{\delta_j}_{k,\mu}(x,0,\omega)).\end{equation}

Using Proposition \ref{sublimit}, since $\omega\in\Omega_3(p,r)$, as $j\rightarrow\infty$, for each $k\in\left\{1,\ldots,m\right\}$, \begin{equation}\label{trio1}\begin{array}{ll} \overline{m}_\mu(x)-m^{\delta_j}_{k,\mu}(x,0,\omega)\rightarrow 0 & \textrm{locally uniformly on}\;\;\;B_r(x_0),\end{array}\end{equation}  using formula (\ref{colmetricformula}), \begin{equation}\label{trio2}\begin{array}{ll} p\cdot x-\overline{m}_\mu(x)\leq 0 & \textrm{and vanishes at}\;\;\; x=x_0.\end{array}\end{equation}  And, since $\omega\in\Omega_3(p,r)\subset\Omega_1$ and by our choice of $\delta_j\rightarrow 0$, for each $k\in\left\{1,\ldots,m\right\}$, $$\lim_{j\rightarrow\infty}(\sup_{x\in\partial \overline{B}_r(x_0)} \delta_j v^{\delta_j}_k(\frac{x}{\delta_j},\omega))=\lim_{j\rightarrow\infty}(\sup_{x\in\partial \overline{B}_R(x_0)} \delta_jw^{\delta_j}_k(\frac{x}{\delta_j},\omega)+\delta_jv^{\delta_j}_1(0,\omega))=-\hat{H}(p,r)$$ and $$\lim_{j\rightarrow\infty}\delta_j v^{\delta_j}_k(\frac{x_0}{\delta_j},\omega)=\lim_{j\rightarrow\infty}\delta_jw^{\delta_j}_k(\frac{x_0}{\delta_j},\omega)+\delta_jv^{\delta_j}_1(0,\omega)=-\hat{H}(p,r).$$  Therefore, \begin{equation}\label{trio3}\limsup_{j\rightarrow\infty}(\sup_{x\in \partial \overline{B}_r(x_0)}\delta_j v^{\delta_j}_k(\frac{x}{\delta_j},\omega)-\eta\abs{x-x_0}^2)=-\hat{H}(p,r)-\eta r^2<-\hat{H}(p,r)=\lim_{j\rightarrow\infty}\delta_j v^{\delta_j}_k(\frac{x_0}{\delta_j},\omega).\end{equation}  It follows from (\ref{minzero5}), (\ref{trio1}), (\ref{trio2}) and (\ref{trio3}) that (\ref{minzero4}) is impossible for large $j$ and, therefore, that $\overline{H}(p,r)=\hat{H}(p,r)$.

We conclude by defining $\Omega_3=\bigcap_{(p,r)\in\mathbb{Q}^n\times\mathbb{Q}}\Omega_3(p,r)$ and applying Proposition \ref{contdep} and (\ref{effcon}).\end{proof}

\section{The Proof of Homogenization and Collapse}

We recall the system \begin{equation}\label{colend} \left\{\begin{array}{ll} u^\epsilon_{k,t}-\epsilon\tr(A_k(\frac{x}{\epsilon},\omega)D^2u^\epsilon_k)+H_k(Du^\epsilon_k,u^\epsilon,\frac{u^\epsilon_k-u^\epsilon_j}{\epsilon},\frac{x}{\epsilon},\omega)=0 & \textrm{on}\;\;\mathbb{R}^n\times(0,\infty),\\ u^\epsilon=u_0 & \textrm{on}\;\;\mathbb{R}^n\times\left\{0\right\},\end{array}\right.\end{equation} for $u_0=(u_{1,0},\ldots,u_{m,0})\in\BUC(\mathbb{R}^n;\mathbb{R}^m)$.   The aim of this section is to prove the homogenization and collapse of (\ref{colend}) to the deterministic scalar equation \begin{equation}\label{eqend}\left\{\begin{array}{ll} \overline{u}_t+\overline{H}(D\overline{u},\overline{u})=0 & \textrm{on}\;\;\;\mathbb{R}^n\times(0,\infty), \\ \overline{u}=\underline{u}_0 & \textrm{on}\;\;\;\mathbb{R}^n\times\left\{0\right\}, \end{array}\right.\end{equation} with $\underline{u}_0$ the point-wise minimum \begin{equation*}\label{inend} \underline{u}_0=\min_{1\leq k\leq m}u_{k,0}.\end{equation*}

The following contraction properties of (\ref{colend}) and (\ref{eqend}) will be used frequently throughout this section.  The proofs follow from elementary methods in the theory of viscosity solutions and are therefore omitted (See \cite{CIL,IK}).

\begin{prop} Assume (\ref{steady}).  Fix $\epsilon>0$.  If $u^\epsilon$ and $v^\epsilon$ are solutions of (\ref{colend}) with initial conditions $u_0$ and $v_0$ respectively, then \begin{equation}\label{contraction}  \max_{1\leq k\leq m}\norm{u^\epsilon_k-v^\epsilon_k}_{L^\infty(\mathbb{R}^n\times[0,\infty))}\leq\max_{1\leq k\leq m}\norm{u_{k,0}-v_{k,0}}_{L^\infty(\mathbb{R}^n)}.\end{equation}

If $u$ and $v$ are solutions of (\ref{eqend}) with initial conditions $u_0$ and $v_0$ respectively, then \begin{equation}\label{homcontraction}\norm{u-v}_{L^\infty(\mathbb{R}^n\times[0,\infty))}\leq\norm{u_0-v_0}_{L^\infty(\mathbb{R}^n)}.\end{equation}\end{prop}

We begin by characterizing the solutions of (\ref{colend}).  The proof of the following proposition follows from standard properties of viscosity solutions and is therefore omitted.

\begin{prop}\label{uniformbound}  Assume (\ref{steady}).  System (\ref{colend}) admits, for each  $\epsilon>0$ and $T>0$, a unique solution $u^\epsilon\in \BUC(\mathbb{R}^n\times[0,T);\mathbb{R}^m)$ such that, for each $T>0$, there exists $C=C(\norm{u_0}_{L^\infty},T)>0$ satisfying \begin{equation*}\label{unibound} \max_{1\leq k\leq m}\norm{u^\epsilon_k}_{L^\infty(\mathbb{R}^n\times[0,T])}\leq C.\end{equation*}\end{prop}

In order to obtain Lipschitz estimates in what follows, we consider \begin{equation}\label{insteady} u_0\in C^{1,1}(\mathbb{R}^n;\mathbb{R}^m).\end{equation}  We account for this in the section's final proposition, where our main result is proven for arbitrary initial data $u_0\in\BUC(\mathbb{R}^n;\mathbb{R}^m)$.

Estimates are first obtained for the $\abs{u^\epsilon_{k,t}}$.  As demonstrated by (\ref{introex}), these estimates cannot be obtained uniformly, as $\epsilon\rightarrow 0$, for general initial conditions.  However, if $u_0$ satisfies \begin{equation}\label{inscalar} \begin{array}{ll} u_{i,0}=u_{j,0} & \textrm{for each}\;\;\;1\leq i,j\leq m, \end{array}\end{equation} then estimates may be obtained uniformly in $\epsilon$.

\begin{prop}\label{time}  Assume (\ref{steady}) and (\ref{insteady}).  There exist constants $C>0$  and $c_\epsilon>0$ satisfying, for each $k\in\left\{1,\ldots,m\right\}$,  \begin{equation*}\label{time1} \begin{array}{ll} -c_\epsilon\leq u^\epsilon_{k,t}\leq C & \textrm{on}\;\;\;\mathbb{R}^n\times[0,\infty).  \end{array}\end{equation*}  If $u_0$ satisfies (\ref{inscalar}) then $c_\epsilon>0$ may be chosen independently of $\epsilon$.\end{prop}

\begin{proof}  The dependence on $\omega\in\Omega$ is suppressed as it plays no role.  Let \begin{equation*}\label{time8}  M_\epsilon=\max_{1\leq k,j\leq m}\norm{\frac{u_{k,0}-u_{j,0}}{\epsilon}}_{L^\infty(\mathbb{R}^n)}.\end{equation*}  In view of (\ref{coercive}) and (\ref{hamincrease}), for each $k\in\left\{1,\ldots,m\right\}$, \begin{equation*}\label{time2}\begin{array}{ll} -C_3\leq H_k(Du_{k,0},u_0,\frac{u_{k,0}-u_{j,0}}{\epsilon},\frac{x}{\epsilon})\leq\norm{H_k(Du_{k,0},u_0,\widehat{M_\epsilon},\frac{x}{\epsilon})}_{L^\infty(\mathbb{R}^n)}<\infty & \textrm{on}\;\;\mathbb{R}^n. \end{array}\end{equation*}  Moreover, (\ref{matsquare}) and (\ref{lipsigma}) yield that there exists $M>0$ satisfying \begin{equation*}\label{time3} \max_{1\leq k\leq m}\norm{\tr(A_kD^2u_{k,0})}_{L^\infty(\mathbb{R}^n)} \leq M. \end{equation*}

Let $C=C_3+M$ and $c_\epsilon=\max_{1\leq k\leq m}\norm{H_k(Du_{k,0},u_0,\widehat{M_\epsilon},\frac{x}{\epsilon})}_{L^\infty(\mathbb{R}^n)}+M.$  It follows that the functions \begin{equation*}\label{time4} \begin{array}{lll} z_+=u_0+\widehat{Ct} & \textrm{and} & z_-=u_0-\widehat{c_\epsilon t} \end{array}\end{equation*} are respectively a supersolution and a subsolution of (\ref{colend}).

The comparison principle implies \begin{equation*}\label{time5}\begin{array}{ll} z_-\leq u^\epsilon\leq z_+ & \textrm{on}\;\;\;\mathbb{R}^n\times[0,\infty),\end{array}\end{equation*} which, in view of (\ref{contraction}), implies \begin{equation*}\label{time7}\begin{array}{ll} -c_\epsilon \leq u^\epsilon_{k,t} \leq C & \textrm{on}\;\;\;\mathbb{R}^n\times[0,\infty).  \end{array} \end{equation*}

If $u_0$ satisfies (\ref{inscalar}), then $M_\epsilon=0$ for each $\epsilon>0$.   Therefore, $c_\epsilon>0$ may be chosen uniformly as $\epsilon\rightarrow 0$.  \end{proof}

We now establish a gradient estimate by Bernstein's method.  The proof is essentially a repetition of the argument appearing in Proposition \ref{cellsol} and, because this fact is not needed in what follows, we omit it.  Notice that in this case, however, the bounds for the $\abs{Du^\epsilon_k}$ depend on the previous estimates for the $\abs{u^\epsilon_{k,t}}$.  We therefore obtain these bounds uniformly, as $\epsilon\rightarrow 0$, only in the case that $u_0$ satisfies (\ref{inscalar}).

\begin{prop}\label{grad}  Assume (\ref{steady}) and (\ref{insteady}).  There exists $C_\epsilon>0$ such that $$\max_{1\leq k\leq m}\norm{Du^\epsilon_k}_{L^\infty(\mathbb{R}^n\times[0,\infty))}\leq C_\epsilon.$$  If $u_0$ satisfies (\ref{inscalar}), then $C_\epsilon$ may be chosen uniformly as $\epsilon\rightarrow 0$. \end{prop}

We now obtain more precise control of the $u^\epsilon_k$'s from above.  This estimate will play an important role in the proof of the collapse of (\ref{colend}).

\begin{prop}\label{altcon}  Assume (\ref{steady}) and (\ref{insteady}).  There exist $C_i=C_i(T)>0$, for $i=1,2$, satisfying, for each $\epsilon>0$ and $k\in\left\{1,\ldots,m\right\}$, $$\begin{array}{ll} u^\epsilon_k-\underline{u}_0\leq C_1(\epsilon+t+e^{-\frac{C_2t}{\epsilon}}) & \textrm{on}\;\;\;\mathbb{R}^n\times[0,T].\end{array}$$\end{prop}

\begin{proof}  The role of $\omega\in\Omega$ is suppressed as it plays no role.  We consider, for each $\epsilon>0$, with $C>0$ as in Proposition \ref{time}, $$M^\epsilon(t)=\max_{1\leq k\leq m}\sup_{(x,y)\in\mathbb{R}^n\times\mathbb{R}^n}(u^\epsilon_k(x,t)-\underline{u}_0(y)-\frac{1}{2\epsilon}\abs{x-y}^2-Ct).$$  We will show that, in the viscosity sense, there exists $\tilde{C}>0$ satisfying, for each $\epsilon>0$, $$M^{\epsilon '} (t)\leq \tilde{C}(1-\frac{M^\epsilon(t)}{\epsilon}).$$

Fix $\epsilon>0$.  Suppose that $M^\epsilon-\phi$ has a strict local maximum at $t_0\in(0,T)$ for $\phi\in \C^2(0,T)$.  We may assume, without loss of generality, that $t_0$ is a strict global maximum.

Define, for each $\eta>0$, \begin{equation*}\label{altcon1} \Phi(x,y,t,k)=u^\epsilon_k(x,t)-\underline{u}_0(y)-\frac{1}{2\epsilon}\abs{x-y}^2-\phi(t)-\frac{\eta}{2}\abs{y}^2.\end{equation*} For all $\eta$ sufficiently small, by Proposition \ref{uniformbound} and because $\underline{u}_0\in\C^{0,1}(\mathbb{R}^n)$, the function $\Phi$ achieves a global maximum at $(x_\eta,y_\eta,t_\eta,k_\eta)$.  Therefore, with the notation of \cite{CIL}, \begin{equation*}\label{altcon2} \left(\phi'(t_\eta),\epsilon^{-1}(x_\eta-y_\eta),\epsilon^{-1}I\right)\in\mathcal{J}^{2,+}u^\epsilon_{k_\eta}(x_\eta,t_\eta),\end{equation*} and, because $u^\epsilon$ is a solution of (\ref{colend}), we have \begin{equation*}\label{altcon3} \phi'(t_\eta)-\tr(A_{k_\eta}(\frac{x_\eta}{\epsilon},\omega))+H_{k_\eta}(\epsilon^{-1}(x_\eta-y_\eta),u^\epsilon,\frac{u^\epsilon_{k_\eta}-u^\epsilon_j}{\epsilon},\frac{x_\eta}{\epsilon},\omega)\leq 0.\end{equation*}

It follows from (\ref{coercive}), (\ref{matsquare}) and (\ref{lipsigma}) that there exists $\tilde{C}>0$ independent of $\epsilon>0$ satisfying \begin{equation*}\label{altcon4} \phi'(t_\eta)\leq \tilde{C}(1-\frac{u^\epsilon_{k_\eta}(x_\eta,t_\eta)-\min_{1\leq i\leq m}u^\epsilon_i(x_\eta,t_\eta)}{\epsilon}).\end{equation*}  Proposition \ref{time} yields, for the same constant $C>0$ independent of $\epsilon>0$, \begin{equation}\label{altcon5} \phi'(t_\eta)\leq \tilde{C}(1-\frac{u^\epsilon_k(x_\eta,t_\eta)-\underline{u}_0(x_\eta)-Ct_\eta}{\epsilon}).\end{equation}  Since $\underline{u}_0\in \C^{0,1}(\mathbb{R}^n)$, we have $\abs{x_\eta-y_\eta}\leq \epsilon \norm{D\underline{u}_0}_{L^\infty(\mathbb{R}^n)}$.  Therefore, for a perhaps larger constant $\tilde{C}>0$ independent of $\epsilon>0$, \begin{equation}\label{altcon6} \phi'(t_\eta)\leq \tilde{C}(1-\frac{u^\epsilon_{k_\eta}(x_\eta,t_\eta)-\underline{u}_0(y_\eta)-Ct_\eta}{\epsilon})\leq\tilde{C}(1-\frac{u^\epsilon_{k_\eta}(x_\eta,t_\eta)-\underline{u}_0(y_\eta)-\frac{1}{2\epsilon}\abs{x_\eta-y_\eta}^2-Ct_\eta}{\epsilon}).\end{equation}

Since $t_0$ is a strict global maximum of $M^\epsilon-\phi$ we have, as $\eta\rightarrow 0$, $$\begin{array}{lll} t_\eta\rightarrow t_0 & \textrm{and} & M^\epsilon(t_0)=\lim_{\eta\rightarrow 0}(u^\epsilon_{k_\eta}(x_\eta,t_\eta)-\underline{u}_0(y_\eta)-\frac{1}{2\epsilon}\abs{x_\eta-y_\eta}^2-Ct_\eta).\end{array}$$  Therefore, by passing to the limit as $\eta\rightarrow 0$ in (\ref{altcon6}), \begin{equation*}\label{altcon7} \phi'(t_0)\leq \tilde{C}(1-\frac{M^\epsilon(t_0)}{\epsilon}).\end{equation*}  This implies that, in the viscosity sense, $$\begin{array}{lllll} M^{\epsilon '}\leq \tilde{C}(1-\frac{M^\epsilon}{\epsilon}) & \textrm{on}\;\;(0,T) & \textrm{and, therefore,} & M^\epsilon(t)\leq \epsilon+(M^\epsilon(0)-\epsilon)e^{-\frac{\tilde{C}t}{\epsilon}} & \textrm{on}\;\;[0,T].\end{array}$$  Let $M_0=\sup_{\epsilon>0}M^\epsilon(0)<\infty$ and conclude, using the definition of $M^\epsilon$, that the proposition holds for $C_1=\max\left\{1,C,M_0\right\}$ and $C_2=\tilde{C}$.\end{proof}

We now obtain more precise control of the $u^\epsilon_k$'s from below.  This estimate and Proposition \ref{altcon} combine to identify the appropriate initial data $\underline{u}_0$.

\begin{prop}\label{altcon1}  Assume (\ref{steady}) and (\ref{insteady}).  For each $\delta>0$ there exists $C_\delta>0$ satisfying, for each $\omega\in\Omega$, $k\in\left\{1,\ldots,m\right\}$ and $\epsilon>0$, $$\begin{array}{ll} u^\epsilon_k\geq \underline{u}_0-\delta-C_\delta t & \textrm{on}\;\;\mathbb{R}^n\times[0,\infty).\end{array}$$\end{prop}

\begin{proof} The role of $\omega\in\Omega$ is suppressed as it plays no role.  Fix $\delta>0$.  Since $\underline{u}_0\in\C^{0,1}(\mathbb{R}^n)$, define by convolution $\underline{u}_0^\delta\in\C^{1,1}(\mathbb{R}^n)$ which, after subtracting a constant, satisfies \begin{equation*}\label{infstarsup5}\begin{array}{ll} \underline{u}_0-\delta\leq \underline{u}_0^\delta\leq \underline{u}_0 & \textrm{on}\;\;\mathbb{R}^n.\end{array}\end{equation*}

Let $u^{\epsilon,\delta}=(u^{\epsilon,\delta}_1,\ldots,u^{\epsilon,\delta}_m)$ denote the solution of (\ref{colend}) corresponding to the initial condition $\widehat{\underline{u}_0^\delta}$.  Since $\widehat{\underline{u}_0^\delta}$ satisfies (\ref{inscalar}) with $\widehat{\underline{u}_0^\delta}\leq u_0$, the comparison principle and Proposition \ref{time} imply that there exists $C_\delta>0$ independent of $\epsilon>0$ satisfying, for each $k\in\left\{1,\ldots,m\right\}$, \begin{equation*}\label{infstarsub7}\begin{array}{ll} u^\epsilon_k\geq u^{\epsilon,\delta}_k\geq\underline{u}_0-\delta-C_\delta t & \textrm{on}\;\;\mathbb{R}^n\times[0,\infty).\end{array}\end{equation*}\end{proof}

Having characterized the solutions $u^\epsilon$ we begin the proof of the main result.  In order to overcome the absence of estimates, uniform in $\epsilon$, for the $\abs{u^\epsilon_{k,t}}$ and $\abs{Du^\epsilon_k}$, we introduce the half-relaxed limits \begin{equation}\label{star} u^*(x,t,\omega)=\limsup_{\epsilon\rightarrow 0}\max_{1\leq k\leq m}\left\{\;u^\epsilon_k(y,s,\omega)\;|\;\abs{y-x}+\abs{s-t}\leq\epsilon\;\right\}\end{equation} and\begin{equation*}\label{star1} u_*(x,t,\omega)=\liminf_{\epsilon\rightarrow 0}\min_{1\leq k\leq m}\left\{\;u^\epsilon_k(y,s,\omega)\;|\;\abs{y-x}+\abs{s-t}\leq\epsilon\;\right\},\end{equation*} where it is understood that, for each $\epsilon>0$, the respective maximum and minimum are taken for $\max(t-\epsilon,0)\leq s\leq t+\epsilon$.

Following standard methods in the theory of viscosity solutions, we wish to prove that $u^*$ and $u_*$ are respectively a subsolution and a supersolution of (\ref{eqend}) on a subset of full probability.  However, as demonstrated by (\ref{introex}) and due to the initial boundary layer, definition (\ref{star}) does not yield a relationship between $u^*$ and $\underline{u}_0$ at $t=0$.  Indeed, if $u_0$ does not satisfy (\ref{inscalar}), Proposition \ref{time} yields $$\begin{array}{ll} u^*=\max_{1\leq k\leq m}u_{k,0}>\min_{1\leq k\leq m}u_{k,0}=\underline{u}_0 & \textrm{on}\;\;\mathbb{R}^n\times\left\{0\right\}.\end{array}$$

To obtain the desired relationship at $t=0$, we use that, in view of Proposition \ref{altcon}, there exists $C>0$ satisfying, for each $\omega\in\Omega$, \begin{equation}\label{starup} u^*\leq \underline{u}_0+Ct\;\;\textrm{on}\;\;\mathbb{R}^n\times(0,1],\end{equation}  and define \begin{equation*}\label{altstar}\tilde{u}^*=\left\{\begin{array}{ll} u^* & \textrm{on}\;\;\;\mathbb{R}^n\times (0,\infty), \\ \underline{u}_0 & \textrm{on}\;\;\;\mathbb{R}^n\times\left\{0\right\}.\end{array}\right.\end{equation*} Observe that, in view of (\ref{starup}), $\tilde{u}^*$ is continuous at $t=0$ and therefore $\tilde{u}^*\in\USC(\mathbb{R}^n\times[0,\infty))$.

We now present, in the final two propositions, the proof of our main result.

\begin{prop}\label{supstarsub} Assume (\ref{steady}) and (\ref{insteady}).  For each $\omega\in\Omega_3$, $\tilde{u}^*$ and $u_*$ are respectively a subsolution and a supersolution of $$\left\{\begin{array}{ll} u_t+\overline{H}(Du,u)=0 & \textrm{on}\;\;\;\mathbb{R}^n\times (0,\infty), \\ u=\underline{u}_0 & \textrm{on}\;\;\;\mathbb{R}^n\times\left\{0\right\}. \end{array}\right.$$\end{prop}

\begin{proof}  Propositions \ref{altcon} and \ref{altcon1} yield that, for each $\omega\in\Omega$ and $\delta>0$, $$\begin{array}{ll} \underline{u}_0-\delta\leq u_*\leq \tilde{u}^*=\underline{u}_0 & \textrm{on}\;\;\;\mathbb{R}^n\times\left\{0\right\}\end{array}$$ and, therefore, $u_*=\tilde{u}^*=\underline{u}_0$ on $\mathbb{R}^n\times\left\{0\right\}$.

The remainder of the proof follows by the standard perturbed test function method.  Fix $\omega\in\Omega_3$.  We argue by contradiction and assume that $\tilde{u}^*-\phi$ has a strict local maximum at $(x_0,t_0)$ for $\phi\in\C^2(\mathbb{R}^n\times(0,\infty))$ satisfying \begin{equation}\label{perturb}\begin{array}{lll} \phi(x_0,t_0)=\tilde{u}^*(x_0,t_0) & \textrm{and} & \phi_t(x_0,t_0)+\overline{H}(D\phi(x_0,t_0),\phi(x_0,t_0))=\delta>0. \end{array}\end{equation}

Define \begin{equation*}\label{perturb1}\begin{array}{lll} \phi^\epsilon=\left(\phi^\epsilon_1,\ldots,\phi^\epsilon_m\right) & \textrm{with} & \phi^\epsilon_k(x,t)=\phi(x,t)+\epsilon w^\epsilon_k(\frac{x}{\epsilon},\omega), \end{array}\end{equation*} where $w^\epsilon$ is the solution of (\ref{cellnorm}) corresponding to the pair $(D\phi(x_0,t_0),\phi(x_0,t_0))$.  Notice that, by Proposition \ref{subsolution}, since $\omega\in\Omega_3\subset\Omega_1$, as $\epsilon\rightarrow0$, \begin{equation}\label{perturb2}\begin{array}{ll} \phi^\epsilon\rightarrow \phi & \textrm{locally uniformly on}\;\;\;\mathbb{R}^n\times(0,\infty).\end{array}\end{equation}  We now prove that there exists $r>0$ such that, for all $\epsilon$ sufficiently small, $\phi^\epsilon$ is a supersolution of (\ref{colend}) on $B_r(x_0)\times(t_0-r,t_0+r)$.

Fix $\epsilon>0$.  Suppose that, for some $k\in\left\{1,\ldots,m\right\}$ and $\eta\in\C^2(\mathbb{R}^n\times(0,\infty))$, $\phi^\epsilon_k-\eta$ has a strict local minimum at $(y_0,s_0)$.  Then, the rescaled function \begin{equation*}\label{perturb3} w^\epsilon_k(y,\omega)-\frac{1}{\epsilon}(\eta(\epsilon y,\epsilon s)-\phi(\epsilon y, \epsilon s))\end{equation*} has a strict local minimum at $(\frac{y_0}{\epsilon},\frac{s_0}{\epsilon})$ with \begin{equation}\label{perturb6} \phi_t(y_0,s_0)=\eta_t(y_0,s_0). \end{equation}  Since $w^\epsilon$ is a solution of (\ref{cellnorm}) we have, after rescaling and evaluated at $(y_0,s_0)$, \begin{equation*}\label{perturb4} \epsilon w^\epsilon_k-\epsilon\tr(A_k(\frac{y_0}{\epsilon},\omega)(D^2\eta-D^2\phi))+H_k(D\eta-D\phi+D\phi(x_0,t_0),\widehat{\phi(x_0,t_0)},\frac{\phi^\epsilon_k-\phi^\epsilon_j}{\epsilon},\frac{y}{\epsilon},\omega)\geq -\epsilon v^\epsilon_1(0).\end{equation*}

It follows from Proposition \ref{colupgrade}, (\ref{perturb}) and because $\omega\in\Omega_3\subset\Omega_1$ that, for all $\epsilon>0$ sufficiently small, $$\begin{array}{lll} \abs{\epsilon w^\epsilon_k(\frac{y_0}{\epsilon},\omega)}<\frac{\delta}{4} & \textrm{and} & \phi_t(x_0,t_0)-\epsilon v^\epsilon_1(0)>\frac{\delta}{2}, \end{array}$$ and, hence, in view of Proposition \ref{cellnormsol}, (\ref{perturb6}), (\ref{matsquare}), (\ref{lipsigma}), (\ref{hamcon}) and because $\phi\in \C^2(\mathbb{R}^n\times (0,\infty))$, there exists $r>0$ independent of $\eta$ such that, for all $\epsilon>0$ sufficiently small, whenever $\abs{x_0-y_0}<r$, \begin{equation*}\label{perturb5} \eta_t-\epsilon\tr(A_k(\frac{y}{\epsilon},\omega)D^2\eta)+H_k(D\eta,\widehat{\phi(x_0,t_0)},\frac{\phi^\epsilon_k-\phi^\epsilon_j}{\epsilon},\frac{y}{\epsilon},\omega)>0. \end{equation*}

We conclude that, for all $\epsilon>0$ sufficiently small, $\phi^\epsilon$ is a supersolution of (\ref{colend}) on $B_r(x_0)\times(t_0-r,t_0+r)$.  The comparison principle yields, for all $\epsilon$ sufficiently small, \begin{equation*}\label{perturb7} \max_{1\leq k\leq m}\max_{\overline{B}_r(x_0)\times[t_0-r,t_0+r]}\left(u^\epsilon_k-\phi^\epsilon_k\right)=\max_{1\leq k\leq m}\max_{\partial(\overline{B}_r(x_0)\times[t_0-r,t_0+r])}\left(u^\epsilon_k-\phi^\epsilon_k\right).\end{equation*}  By (\ref{perturb2}) and standard optimization results (See \cite{CIL}), this contradicts the assumption that $\tilde{u}^*-\phi$ has a strict local maximum at $(x_0,t_0)$.

The proof that $u_*$ is a supersolution is analogous.\end{proof}

We now present our main result.  In the proof, we first deduce the result for initial conditions $u_0\in\C^{1,1}(\mathbb{R}^n;\mathbb{R}^m)$ and then prove the result for general $u_0\in\BUC(\mathbb{R}^n;\mathbb{R}^m)$ by approximation.

\begin{prop}\label{homogenization}  Assume (\ref{steady}).  For each $\omega\in\Omega_3$, the solutions $u^\epsilon$ of (\ref{colend}) satisfy, as $\epsilon\rightarrow 0$, for each $k\in\left\{1,\ldots,m\right\}$, \begin{equation}\label{homogenization1}\begin{array}{ll} u^\epsilon_k\rightarrow \overline{u} & \textrm{locally uniformly on}\;\;\;\mathbb{R}^n\times(0,\infty),\end{array}\end{equation} for $\overline{u}$ the solution of the scalar equation \begin{equation}\label{homogenization2} \left\{\begin{array}{ll} \overline{u}_t+\overline{H}(D\overline{u},\overline{u})=0 & \textrm{on}\;\;\;\mathbb{R}^n\times(0,\infty), \\ \overline{u}=\underline{u}_0 & \textrm{on}\;\;\;\mathbb{R}^n\times\left\{0\right\}.\end{array}\right.\end{equation} \end{prop}

\begin{proof}  Suppose that $u_0\in\C^{1,1}(\mathbb{R}^n;\mathbb{R}^m)$.  Then, in view of Proposition \ref{supstarsub} and the comparison principle, $\tilde{u}^*\in\USC(\mathbb{R}^n\times[0,\infty))$ and $u_*\in\LSC(\mathbb{R}^n\times[0,\infty))$ satisfy, for each $\omega\in\Omega_3$, $$\begin{array}{ll} \tilde{u}^*\leq u_* & \textrm{on}\;\;\;\mathbb{R}^n\times[0,\infty).\end{array}$$  Since the opposite inequality follows by definition, we conclude that, for each $\omega\in\Omega_3$, $\tilde{u}^*=u_*$ is the unique solution of (\ref{homogenization2}) which, by standard properties of viscosity solutions, implies (\ref{homogenization1}).

We now consider $u_0\in\BUC(\mathbb{R}^n;\mathbb{R}^m)$.  Define by convolution functions $u^\delta_0=(u^\delta_{1,0},\ldots,u^\delta_{m,0})\in\C^{1,1}(\mathbb{R}^n;\mathbb{R}^m)$ satisfying $$\max_{1\leq k\leq m}\norm{u^\delta_{k,0}-u_{k,0}}_{L^\infty(\mathbb{R}^n)}<\delta.$$  Let $u^\epsilon$ be the solution of (\ref{colend}) with initial condition $u_0$ and let $u^{\epsilon,\delta}$ be the solution with initial condition $u^\delta_0$.  Similarly, let $\overline{u}$ be the solution of (\ref{eqend}) with initial condition $\underline{u}_0$ and let $\overline{u}^\delta$ be the solution with initial condition $\underline{u}^\delta_0$.

Fix a compact set $K\subset\mathbb{R}^n\times(0,\infty)$, $\eta>0$ and $\omega\in\Omega_3$.  We have $$\max_{1\leq k\leq m}\norm{u^\epsilon_k-\overline{u}}_{L^\infty(K)}\leq \max_{1\leq k\leq m}(\norm{u^\epsilon_k-u^{\epsilon,\delta}_k}_{L^\infty(K)}+\norm{u^{\epsilon,\delta}_k-\overline{u}^\delta}_{L^\infty(K)}+\norm{\overline{u}^\delta-\overline{u}}_{L^\infty(K)})$$ which, in view of (\ref{contraction}) and (\ref{homcontraction}), yields, for $\delta=\frac{\eta}{2}$, $$\max_{1\leq k\leq m}\norm{u^\epsilon_k-\overline{u}}_{L^\infty(K)}< \eta+\max_{1\leq k\leq m}\norm{u^{\epsilon,\frac{\eta}{2}}_k-\overline{u}^{\frac{\eta}{2}}}_{L^\infty(K)}.$$

Since, for each $\delta>0$, $u^\delta_0\in \C^{1,1}(\mathbb{R}^n;\mathbb{R}^m)$, the above implies $$\begin{array}{lll} \lim_{\epsilon\rightarrow 0}\max_{1\leq k\leq m}\norm{u^{\epsilon,\frac{\eta}{2}}_k-\overline{u}^\frac{\eta}{2}}_{L^\infty(K)}=0 & \textrm{and, hence,} & \limsup_{\epsilon\rightarrow 0}\max_{1\leq k\leq m}\norm{u^\epsilon_k-\overline{u}}_{L^\infty(K)}<\eta.\end{array}$$  Because $K\subset\mathbb{R}^n\times(0,\infty)$, $\eta>0$ and $\omega\in\Omega_3$ were arbitrary this completes the proof.  \end{proof}

\section{Variations}

We conclude the paper with a summary of the results for related systems.  Because the proofs are either simpler or identical to those presented above we omit them.  The system \begin{equation}\label{endnocol}\left\{\begin{array}{ll} u^\epsilon_{k,t}-\epsilon\tr(A_k(\frac{x}{\epsilon},\omega)D^2\ue)+H_k(D\ue,u^\epsilon,\frac{x}{\epsilon},\omega)=0 & \textrm{on}\;\;\; \mathbb{R}^n\times(0,\infty), \\ u^\epsilon=u_0 & \textrm{on}\;\;\;\mathbb{R}^n\times\left\{0\right\},\end{array}\right.\end{equation} where the coefficients satisfy assumptions identical to those presented in Section 2, differs from (\ref{colend}) in that the Hamiltonians are no longer coercive in the differences $\epsilon^{-1}(u^\epsilon_k-u^\epsilon_j)$.  The result is therefore the homogenization of (\ref{endnocol}) to a deterministic system.

\begin{prop}\label{nocolresult}  Assume (\ref{steady}).  There exist deterministic Hamiltonians $\overline{H}_k(p,r)$ and $\Omega'\subset\Omega$ of full probability such that, for each $\omega\in\Omega'$, the solutions $u^\epsilon$ of (\ref{endnocol}) satisfy, as $\epsilon\rightarrow 0$, for each $k\in\left\{1,\ldots,m\right\}$, $$\begin{array}{ll} u^\epsilon_k\rightarrow\overline{u}_k & \textrm{locally uniformly on}\;\;\;\mathbb{R}^n\times[0,\infty),\end{array}$$ for $\overline{u}=(\overline{u}_1,\ldots,\overline{u}_m)$ the solution of the system $$\left\{\begin{array}{ll} \overline{u}_{k,t}+\overline{H}_k(D\overline{u}_k,\overline{u})=0 & \textrm{on}\;\;\;\mathbb{R}^n\times(0,\infty),\\ \overline{u}=u_0 & \textrm{on}\;\;\;\mathbb{R}^n\times\left\{0\right\}.\end{array}\right.$$\end{prop}

Our methods also apply to the analogous time-independent systems, \begin{equation}\label{colnot} \begin{array}{ll} -\epsilon\tr(A_k(\frac{x}{\epsilon},\omega)D^2\ue)+H_k(D\ue,u^\epsilon,\frac{u^\epsilon_k-u^\epsilon_j}{\epsilon},\frac{x}{\epsilon},\omega)=0 & \textrm{on}\;\;\; \mathbb{R}^n,\end{array}\end{equation}
and, \begin{equation}\label{nocolnot}\begin{array}{ll} -\epsilon\tr(A_k(\frac{x}{\epsilon},\omega)D^2\ue)+H_k(D\ue,u^\epsilon,\frac{x}{\epsilon},\omega)=0 & \textrm{on}\;\;\; \mathbb{R}^n,\end{array}\end{equation} where the coefficients satisfy assumptions identical to those presented in Section 2.  Furthermore, we assume that there exists $\lambda>0$ such that, for $r,q\in\mathbb{R}^m$, if $r_k-q_k=\max_{1\leq i\leq m}\abs{r_i-q_i}$ then, for each $(p,s,y,\omega)$, \begin{equation}\label{notstrict} H_k(p,r,s,y,\omega)-H_k(p,q,s,y,\omega)\geq\lambda(r_k-q_k). \end{equation}  The following two propositions summarize, respectively, the main results for (\ref{colnot}) and (\ref{nocolnot}).

\begin{prop}\label{colnotresult}  Assume (\ref{steady}) and (\ref{notstrict}).  There exists a deterministic Hamiltonian $\overline{H}(p,r)$ and $\Omega'\subset\Omega$ of full probability such that, for each $\omega\in\Omega'$, the solutions $u^\epsilon$ of (\ref{colnot}) satisfy, as $\epsilon\rightarrow 0$, for each $k\in\left\{1,\ldots,m\right\}$, $$\begin{array}{ll} u^\epsilon_k\rightarrow\overline{u} & \textrm{locally uniformly on}\;\;\;\mathbb{R}^n,\end{array}$$ for $\overline{u}$ the solution of the scalar equation $$\begin{array}{ll}\overline{H}(D\overline{u},\overline{u})=0 & \textrm{on}\;\;\;\mathbb{R}^n.\end{array}$$\end{prop}

\begin{prop}\label{nocolnotresult}  Assume (\ref{steady}) and (\ref{notstrict}).  There exist deterministic Hamiltonians $\overline{H}_k(p,r)$ and $\Omega'\subset\Omega$ of full probability such that, for each $\omega\in\Omega'$, the solutions $u^\epsilon$ of (\ref{nocolnot}) satisfy, as $\epsilon\rightarrow 0$, for each $k\in\left\{1,\ldots,m\right\}$, $$\begin{array}{ll} u^\epsilon_k\rightarrow\overline{u}_k & \textrm{locally uniformly on}\;\;\;\mathbb{R}^n,\end{array}$$ for $\overline{u}=(\overline{u}_1,\ldots,\overline{u}_m)$ the solution of the system $$\begin{array}{ll}\overline{H}_k(D\overline{u}_k,\overline{u})=0 & \textrm{on}\;\;\;\mathbb{R}^n.\end{array}$$\end{prop}

\bibliography{hjsystem}
\bibliographystyle{plain}

\end{document}